\documentclass[11pt]{amsart}

\usepackage{amssymb}
\usepackage{amsthm}
\usepackage{amsmath}
\usepackage[mathscr]{eucal}
\usepackage{mathrsfs}
\usepackage{psfrag}
\usepackage{fullpage}
\usepackage{color}
\usepackage{eucal}
\usepackage[dvips]{graphicx}
\usepackage{amsfonts}
\usepackage{lipsum}

\providecommand{\U}[1]{\protect\rule{.1in}{.1in}}

\newtheorem{proposition}{Proposition}[section]
\newtheorem{theorem}[proposition]{Theorem}
\newtheorem{corollary}[proposition]{Corollary}
\newtheorem{lemma}[proposition]{Lemma}

\newtheorem{definition}[proposition]{Definition}
\newtheorem{remark}[proposition]{Remark}

\newtheorem{condition}[proposition]{Condition}

\numberwithin{equation}{section}
\numberwithin{proposition}{section}

\makeatletter
\g@addto@macro{\endabstract}{\@setabstract}
\newcommand{\authorfootnotes}{\renewcommand\thefootnote{\@fnsymbol\c@footnote}}%
\makeatother

\begin{document}

\begin{center}
  \LARGE
  Maximum Likelihood Estimation for Small Noise Multiscale Diffusions \par \bigskip

  \normalsize
  \authorfootnotes
  Konstantinos Spiliopoulos
\textsuperscript{1},
Alexandra Chronopoulou
\textsuperscript{2} \par \bigskip

  \textsuperscript{1}Department of Mathematics \& Statistics,
Boston University\\
111 Cummington Street, Boston MA 02215, e-mail: kspiliop@math.bu.edu \par
  \textsuperscript{2}Department of Statistics and Applied Probability,
 University of California, Santa Barbara\\
Santa Barbara, CA 93106-3110, e-mail: chronopoulou@pstat.ucsb.edu \par \bigskip

  \today
\end{center}

\date{\today}


\begin{abstract}
We study the problem of parameter estimation for stochastic
differential equations with small noise and fast oscillating
parameters. Depending on how fast the intensity of the noise goes to
zero relative to the homogenization parameter, we consider three
different regimes. For each regime, we construct the maximum likelihood estimator
and we study its consistency and asymptotic normality properties.
A simulation study for the first order Langevin equation with a two scale potential is  also provided.
\end{abstract}



\noindent \textbf{Acknowledgement:} The authors would like to thank the anonymous reviewer for pointing out a gap  in the proof of Theorem 5.1 in the original article, as well as all  comments that lead to a significant improvement of the article. K.S. was partially supported, during revisions of this article, by the National Science Foundation
(DMS 1312124).


\section{Introduction}
Data obtained from a physical system sometimes possess many
characteristic length and time scales. In such cases,
it is desirable to construct models that are effective for
large-scale structures, whilst capturing small scales at the same
time. Modeling this type of data
via diffusion type models may be well-suited in many cases.  Thus, multiscale diffusion models have been used to describe
the behavior of physical phenomena in  scientific areas such as chemistry and biology \cite{CPV2010,Janke, PaviotisStuart2007, Zwanzig},
ocean-atmosphere sciences \cite{MFK2008}, finance and econometrics
\cite{AMZ2,FouquePapanicolaouSricar2000}. In many of these
problems, the noise is taken to be small because one may, for example, be
interested in modeling (a):  rare transition events between equilibrium
states of a rough energy landscape
\cite{DupuisSpiliopoulosWang2, Janke,Zwanzig}, or (b):
short time maturity asymptotics for fast mean reverting stochastic
volatility models \cite{FengFordeFouque, FengFouqueKumar}. See also \cite{FreidlinWentzell88,Kutoyants1994} for a thorough discussion on different mathematical and statistical modeling aspects of perturbations of dynamical systems by small noise.

Parameter estimation in multiscale models with small noise is a problem of great practical
importance, due to their wide range of applications, but also of great difficulty, due to the different separating scales. The goal of
this paper is to develop a theoretical framework for the estimation of unknown parameters in a multiscale diffusion
model with vanishing noise. More specifically, let $T>0$ be given
and consider the $d$-dimensional process
$X^{\epsilon}\doteq\{X_{t}^{\epsilon},0\leq t\leq T\}$ satisfying
the stochastic differential equation (SDE)
\begin{equation}
dX_{t}^{\epsilon}=\left[  \frac{\epsilon}{\delta}b_{\theta}\left(  X_{t}^{\epsilon
},\frac{X_{t}^{\epsilon}}{\delta}\right)  +c_{\theta}\left(  X_{t}^{\epsilon}%
,\frac{X_{t}^{\epsilon}}{\delta}\right)  \right]  dt+\sqrt{\epsilon}%
\sigma\left(  X_{t}^{\epsilon},\frac{X_{t}^{\epsilon}}{\delta}\right)
dW_{t},\hspace{0.2cm}X_{0}^{\epsilon}=x_{0}, \label{Eq:LDPandA1}%
\end{equation}
where $\delta=\delta(\epsilon)\downarrow0$ as $\epsilon\downarrow0$, $\theta\in\Theta\subset \mathbb{R}^{p}$ is an unknown parameter and $W_{t}$ is a standard $d$-dimensional Wiener process. The functions $b_{\theta}(x,y),c_{\theta}(x,y)$ and $\sigma(x,y)$ are assumed to be smooth, in the sense of Condition \ref{A:Assumption1}, and periodic with period $\lambda$ in every direction with respect to the second variable.

The rate of convergence of $\delta$ and $\epsilon$ to zero determines the type of equation that one obtains in the limit. For example, if $\delta$ is of order 1 as $\epsilon$ goes to zero, then equation (\ref{Eq:LDPandA1}) reduces to a deterministic ODE that we obtain if we set $\epsilon$ equal to zero. On the other hand, if $\epsilon$ is of order 1 as $\delta$ goes to zero, then homogenization occurs and this results to an equation with homogenized coefficients. When both parameters $\epsilon$ and $\delta$ go to zero together, then we need to consider three different regimes depending on how fast $\epsilon$ goes to zero relative to $\delta$:
\begin{equation}
\lim_{\epsilon\downarrow0}\frac{\epsilon}{\delta}=%
\begin{cases}
\infty & \text{Regime 1,}\\
\gamma\in(0,\infty) & \text{Regime 2,}\\
0 & \text{Regime 3.}%
\end{cases}
\label{Def:ThreePossibleRegimes}%
\end{equation}

We mention here that asymptotic problems for models like
(\ref{Eq:LDPandA1}) have a long history in the mathematical
literature. We refer the interested reader to classical manuscripts
such as \cite{BLP, FreidlinWentzell88, PS} for averaging and
homogenization results and to the more recent articles \cite{DupuisSpiliopoulos,FS} for large
deviations results and
\cite{DupuisSpiliopoulosWang,DupuisSpiliopoulosWang2} for importance
sampling results on related rare event estimation problems.

In (\ref{Eq:LDPandA1}) we assume that the drift term, through the functions $b_{\theta}$ and $c_{\theta}$, depends on a physical parameter $\theta$.  Generally, from a statistical inference point of view, the main questions of interest are the following:
\begin{enumerate}
\item{How can one estimate the fast oscillating parameter $\delta$ and the intensity of the noise $\epsilon$?}
\item{How can one estimate the unknown parameter $\theta?$}
\end{enumerate}
The first question is undoubtedly a quite difficult one  and is not
addressed in the current work; see \cite{Papavasileiou2010} for some
related results for specific equations and further references.
Instead, we focus on the second question. Thus, assuming that the regime
of interaction between $\epsilon$ and $\delta$ is known, we want to
estimate the unknown parameter $\theta$ at time $T$,  based on the continuously observed process $X^{\epsilon}$ up to this time.

In order to do so, we will follow  the maximum likelihood method.  Maximum likelihood estimation in multiscale diffusions with noise of order  $O(1)$
has been studied by different authors and under different settings, see for example
\cite{AzencottBeriTimofeyev2010a, AzencottBeriTimofeyev2010b,KrumscheidPavliotisKalliadasis2011, PaviotisStuart2007, PapavasileiouPaviotisStuart2009}.
We also refer the reader to the manuscripts \cite{Bishwal, Kutoyants2004,Rao} for general results on statistical estimation for diffusion processes.
The novelty of the present paper stems from the fact that we address the problem of parameter estimation when both multiscale effects and
small noise are present, for all three regimes in (\ref{Def:ThreePossibleRegimes}), which requires a different approach for the construction of maximum likelihood estimators.

Indeed, in \cite{PaviotisStuart2007, PapavasileiouPaviotisStuart2009},  assuming that the noise is of order $O(1)$,
the authors fit the data from the prelimit process to the log-likelihood function of the limiting process, i.e., of the process to which
$X^{\epsilon=1,\delta}$ converges to, as $\delta\downarrow 0$. However, when the diffusion coefficient vanishes in the limit, the limiting process is no longer the solution of an SDE, but of an ODE (see Theorem \ref{T:LLN}), thus it is deterministic and does not have a well defined likelihood.
Therefore, instead of working with the likelihood function of the limiting process, we work with the log-likelihood of the original
multiscale model and we infer consistency and asymptotic normality (under conditions as described below)
by studying its limit.



In particular, under Regime 1 with $b=0$ and under Regimes $2$ and $3$ (see (\ref{Def:ThreePossibleRegimes})),
we prove that the maximum likelihood estimator  (MLE) is consistent and asymptotically normal under
broad conditions. The situation of Regime $1$ with $b\neq 0$ is more
complicated, because the original log-likelihood function does not
have a well defined limit as $\epsilon\downarrow 0$, due to the
$\epsilon/\delta\uparrow\infty$ terms. We address this issue by
introducing a modified (pseudo) log-likelihood which
is well defined in the limit. It turns out that the resulting pseudo MLE is not consistent,
however its ``bias'' can be computed exactly.
This is a known problem in multiscale parameter estimation problems \cite{AMZ2, AzencottBeriTimofeyev2010a, AzencottBeriTimofeyev2010b,KrumscheidPavliotisKalliadasis2011, PaviotisStuart2007, PapavasileiouPaviotisStuart2009};
see Section  \ref{S:MLE} for some more details on this.
\begin{remark}
In this article, by ``bias'' we mean the remainder term when we compute the limit of the estimator in probability, that is $\hat{\theta}_{\epsilon} \rightarrow ^{in\;P} \theta + \text{ bias }$. The reason why we use quotes is because bias is usually defined as the remainder of the $L^{2}$-limit of the estimator.
\end{remark}
Under Regime 1 with $b\neq 0$, we support our
findings with a simulation study for a small noise diffusion in a
two-scale potential field, a model of interest in the physical
chemistry literature, \cite{DupuisSpiliopoulosWang2, Janke,PS,
Zwanzig}. For this particular model, we can construct
an estimator that is consistent and normal.

The rest of the paper is organized as follows. In Section \ref{S:Prelim}, we
establish the necessary notation and we present the main ingredients
and assumptions needed in the sequel. In Section \ref{S:MLE} we discuss the maximum likelihood estimation problem for all three
regimes. For Regimes 2 and 3 and Regime 1 when $b=0$, we prove the consistency of the MLE, studying the limit of the log-likelihood function,
in Section \ref{S:LimitngLikelihood}, whereas 
we prove a central limit theorem for the MLE in Section \ref{S:CLT}.  Finally, in Section \ref{S:Examples} we study a
particularly interesting case for Regime $1$, when $b\neq 0$; a
small noise diffusion in a two-scale potential field, we prove a
central limit theorem for the pseudo MLE in this particular setup
and we present a simulated study illustrating the theoretical findings.


\section{Preliminaries, notation and assumptions}\label{S:Prelim}

We work with the canonical filtered probability space $(\Omega,\mathfrak{F}%
,\mathbb{P}_{\theta})$ equipped with a filtration $\mathfrak{F}_{t}$ that satisfies the usual conditions, namely, $\mathfrak{F}_{t}$ is right continuous and $\mathfrak{F}_{0}$ contains all $\mathbb{P}_{\theta}$-negligible sets.

Regarding the SDE (\ref{Eq:LDPandA1}) we impose the following condition.
\begin{condition}
\label{A:Assumption1}

\begin{enumerate}
\item{The parameter $\theta\in\Theta\subset \mathbb{R}^{p} $ where $\Theta$ is open, bounded and convex. Also, the coefficients $b_{\theta}(x,y),c_{\theta}(x,y)$ are Lipschitz continuous in $\theta$. }

\item The functions $b_{\theta}(x,y),c_{\theta}(x,y),\sigma(x,y)$ are Lipschitz continuous and
bounded in both variables. Moreover, they are periodic with period
$\lambda$ in the second variable in each direction. In the case of
Regime $1$ we additionally assume that they are
$C^{1}(\mathbb{R}^{d})$ in $y$ and $C^{2}(\mathbb{R}^{d})$ in $x$
with all partial derivatives continuous and globally bounded in $x$
and $y $.

\item The diffusion matrix $\sigma\sigma^{T}$ is uniformly nondegenerate.

\end{enumerate}
\end{condition}
For notational convenience we define the
operator $\cdot:\cdot$, where for two matrices $A=[a_{ij}],B=[b_{ij}]$
\[
A:B\doteq\sum_{i,j}a_{ij}b_{ij}.
\]
Under Regime $1$, we also impose the following condition.

\begin{condition}
\label{A:Assumption2} Consider the second order elliptic partial
differential operator
\begin{equation*}
\mathcal{L}_{x,\theta}^{1}=b_{\theta}(x,y)\cdot\nabla_{y}+\frac{1}{2}\sigma(x,y)\sigma
(x,y)^{T}:\nabla_{y}\nabla_{y} \label{OperatorRegime1}%
\end{equation*}
equipped with periodic boundary conditions in $y$ ($x$ is being
treated as a parameter here). Let $\mu^{1}_{\theta}(dy;x)$ be the
unique invariant measure corresponding to the operator
$\mathcal{L}_{x,\theta}^{1}$.
 Under Regime 1, we assume the standard centering condition
(see \cite{BLP}) for the drift term $b$:
\[
\int_{\mathcal{Y}}b_{\theta}(x,y)\mu^{1}_{\theta}(dy;x)=0, \text{
for all }\theta\in\Theta,
\]
where $\mathcal{Y}=\mathbb{T}^{d}$ denotes the $d$-dimensional torus.
\end{condition}

Under Conditions \ref{A:Assumption1} and \ref{A:Assumption2},
Theorem 3.3.4 in \cite{BLP} guarantees that for each
$\ell\in\{1,\ldots,d\}$ there is a unique, twice differentiable
function $\chi_{\ell}(x,y)$ that is one periodic in every direction
in $y$ and which solves the following cell problem:
\begin{equation}
\mathcal{L}_{x,\theta}^{1}\chi_{\ell,\theta}(x,y)=-b_{\ell,\theta}(x,y),\quad\int_{\mathcal{Y}}%
\chi_{\ell,\theta}(x,y)\mu^{1}_{\theta}(dy;x)=0. \label{Eq:CellProblem}%
\end{equation}
We write $\chi_{\theta}=(\chi_{1,\theta},\ldots,\chi_{d,\theta})$.

Under Regime $3$, we also impose the following condition.

\begin{condition}
\label{A:Assumption3} Under Regime $3$ and for any $\theta\in\Theta$
and $x\in\mathbb{R}^{d}$, we assume that the ordinary differential
equation
\begin{equation}
\dot{z}_{t}=c_{\theta}(x,z_{t})
\end{equation}
has a unique invariant measure that is Lipschitz continuous in $(\theta,x)\in\Theta\times\mathbb{R}^{d}$.
\end{condition}

Notice that the existence of a unique smooth invariant measure is immediately implied for Regimes $1$ and $2$ due to Condition \ref{A:Assumption1}. However, the situation is more complicated for Regime $3$, since the operator of interest is a first order operator, where clearly the non-degeneracy condition does not hold. For example Condition \ref{A:Assumption3} certainly holds in dimension $d=1$, when $c_{\theta}(x,y)>0$ everywhere and it is sufficiently smooth.

Before stating the main results, we need additional notation and definitions.
We borrow some notation from \cite{DupuisSpiliopoulos} and modify it to fit our needs.
\begin{definition}
\label{Def:ThreePossibleOperators} For the three possible Regimes $i=1,2,3$
defined in (\ref{Def:ThreePossibleRegimes}) and for $x\in\mathbb{R}^{d}%
,y\in\mathcal{Y}$, let
\begin{align}
\mathcal{L}_{x,\theta}^{1}  &  =b_{\theta}(x,y)\cdot\nabla_{y}+\frac{1}{2}\sigma
(x,y)\sigma(x,y)^{T}:\nabla_{y}\nabla_{y},\nonumber\\
\mathcal{L}_{x,\theta}^{2}  &  =\left[  \gamma b_{\theta}(x,y)+c_{\theta}(x,y)\right]
\cdot\nabla_{y}+\gamma\frac{1}{2}\sigma(x,y)\sigma(x,y)^{T}:\nabla_{y}%
\nabla_{y},\nonumber\\
\mathcal{L}_{x,\theta}^{3}  &  = c_{\theta}(x,y)  \cdot\nabla
_{y}.\nonumber
\end{align}
For $i=1,2$ we let $\mathcal{D}(\mathcal{L}_{x,\theta}^{i})=\mathcal{C}%
^{2}(\mathcal{Y})$ and for $i=3$, $\mathcal{D}(\mathcal{L}_{x,\theta}%
^{3})=\mathcal{C}^{1}(\mathcal{Y})$.
\end{definition}
We also define for Regime $i$ a function $\lambda_{i}(x,y)$, $i=1,2,3$, as follows.
\begin{definition}
\label{Def:ThreePossibleFunctions} For the three possible Regimes $i=1,2,3$
defined in (\ref{Def:ThreePossibleRegimes}) and for $x\in\mathbb{R}^{d}
,y\in\mathcal{Y}$, define $\lambda_{i}(x,y):\mathbb{R}%
^{d}\times\mathcal{Y}\rightarrow\mathbb{R}^{d}$ by
\begin{align}
\lambda_{1,\theta}(x,y)  &  =\left(  I+\frac{\partial\chi_{\theta}}{\partial y}(x,y)\right)
 c_{\theta}(x,y), \nonumber\\
\lambda_{2,\theta}(x,y)  &  =\gamma b_{\theta}(x,y)+c_{\theta}(x,y),\nonumber\\
\lambda_{3,\theta}(x,y)  &  =c_{\theta}(x,y),\nonumber
\end{align}
where $\chi_{\theta}=(\chi_{1,\theta},\ldots,\chi_{d,\theta})$ is defined by (\ref{Eq:CellProblem})
and $I$ is the identity matrix.
\end{definition}

Based on the results in \cite{DupuisSpiliopoulos}, we obtain the following theorem, which essentially is the law of large numbers for (\ref{Eq:LDPandA1}). Given the results in \cite{DupuisSpiliopoulos}, the additional steps required in order to prove this theorem are minimal, so we include a short proof in this section as well.
\begin{theorem}
\label{T:LLN} Consider any $x_{0}\in\mathbb{R}^{d}$ and any $T>0$. Assume Condition \ref{A:Assumption1}. In addition, in Regime 1 assume
Condition \ref{A:Assumption2} and under Regime 3 assume Condition \ref{A:Assumption3}. Then, for all $\theta\in\Theta$ and $\eta>0$ and for Regime $i=1,2,3$, we have
\begin{equation*}
\lim_{\epsilon\downarrow 0}\mathbb{P}_{\theta}\left[\sup_{0\leq t\leq T}\left|X^{\epsilon}_{t}-\bar{X}^{i}_{t}\right|>\eta\right]=0,
\end{equation*}
where for Regime $i$, $\bar{X}^{i}$ is the unique solution to the deterministic equation
\begin{equation}
\bar{X}^{i}_{t}=x_{0}+\int_{0}^{t}\int_{\mathcal{Y}}\lambda_{i,\theta}(\bar{X}^{i}_{s},y)\mu^{i}_{\theta}(dy;\bar{X}^{i}_{s})ds\label{Eq:LimitingODE}
\end{equation}
and $\mu^{i}_{\theta}(dy;x)$ is the invariant measure corresponding to the operator $\mathcal{L}_{x}^{i}$ from Definition \ref{Def:ThreePossibleOperators}.
\end{theorem}
\begin{proof}
Under our assumptions, Theorem 2.8 in \cite{DupuisSpiliopoulos} guarantees weak convergence of $X^{\epsilon}_{\cdot}$ to $\bar{X}^{i}_{\cdot}$ in $\mathcal{C}([0,T])$ for any $T>0$.  Since, the limiting process $\bar{X}^{i}_{t}$ is deterministic and weak convergence to constants implies convergence in probability, we obtain the claim of the theorem. Also, due to our assumptions, the limiting ODE's in (\ref{Eq:LimitingODE}) are well defined and have a unique solution in their corresponding regime.
\end{proof}


\section{Maximum likelihood estimation}\label{S:MLE}

Assume that we observe the process $X^{\epsilon}$ in continuous time and denote by  $\mathcal{X}_{T}\doteq \left\{x_{t},0\leq t\leq T\right\}$  the data we obtain. The log-likelihood function for estimating the parameter $\theta$ in the statistical model (\ref{Eq:LDPandA1}) can be expressed as follows
\begin{equation}
Z_{\theta,T}^{\epsilon}(\mathcal{X}_{T})=
\int_{0}^{T}\left<\frac{\epsilon}{\delta}b_{\theta}+c_{\theta}, dx_{s}\right>_{\alpha}\left(x_{s},\frac{x_{s}}{\delta}\right)  -
\frac{1}{2}\int_{0}^{T}\left\Vert \frac{\epsilon}{\delta}b_{\theta}+c_{\theta}\right\Vert^{2}_{\alpha}\left(x_{s},\frac{x_{s}}{\delta}\right) ds,
\label{Eq:LikelihoodFunction}
\end{equation}
where we denote $\alpha(x,y)=\sigma\sigma^{T}(x,y)$ and for any positive definite matrix $K$
\begin{equation*}
\left<p,q\right>_{K}\doteq\left(K^{-1/2}p,K^{-1/2}q\right)\quad \textrm{ and } \left\Vert p\right\Vert_{K}^{2}\doteq\left<p,p\right>_{K}
\end{equation*}
\begin{remark}
The notation used in (\ref{Eq:LikelihoodFunction}) is slightly unusual and the brackets $(x_{t},x_{t}/\delta)$ outside of the integral are the integrand variables. This notation is chosen for presentation purposes only, since if we used the arguments to each  function in the stochastic integrals, this would result in long and complicated-looking formulas.
\end{remark}
Sometimes, we will omit the subscript $K$ if $K=I$. Essentially, we define the likelihood function as the Radon-Nikodym derivative
\begin{equation*}
\frac{d\mathbb{P}_{\theta}}{d\mathbb{P}_{\theta}^{*}} = \exp{\{\frac{1}{\epsilon}Z_{\theta,T}^{\epsilon}(\mathcal{X}_{T})\}},
\end{equation*}
where $\mathbb{P}_{\theta}$ is the measure for (\ref{Eq:LDPandA1}) and $\mathbb{P}_{\theta}^{*}$ the measure for (\ref{Eq:LDPandA1}) when the drift term is equal to zero.
Therefore, for fixed $\epsilon,\delta$, we define the maximum likelihood estimator (MLE) of $\theta$ to be
\begin{equation*}
\hat{\theta}^{\epsilon}\doteq\text{argmax}_{\theta\in\Theta} Z_{\theta,T}^{\epsilon}(\mathcal{X}_{T}).
\end{equation*}

The presence of the small parameters $\epsilon$ and $\delta$ complicate the estimation of $\theta$ significantly. Our approach is to find the limiting likelihood (in the appropriate sense) for each Regime $i=1,2,3$, that is
\begin{equation*}
\bar{Z}^{i}_{\theta,T}(\bar{X}^{i}_{\cdot})=\lim_{\epsilon\downarrow 0}Z_{\theta,T}^{\epsilon}(\mathcal{X}_{T}).
\end{equation*}
Then, we prove consistency and derive asymptotic properties of the MLE $\hat{\theta}^{\epsilon}$, by studying properties of the prelimiting log-likelihood $Z_{\theta,T}^{\epsilon}(\mathcal{X}_{T})$ and of the limiting log-likelihood $\bar{Z}^{i}_{\theta,T}(\bar{X}^{i}_{\cdot})$.

In particular, as we shall see in Section \ref{S:Averaging}, based on the analysis of the log-likelihood function (\ref{Eq:LikelihoodFunction}) we  prove that the MLE is a consistent estimator of the true value $\theta_{0}$, under Regime $1$ with $b=0$ and  Regimes $2$ and $3$. Under the same framework, we also prove, in Section \ref{S:CLT}, that the MLE  $\hat{\theta}^{\epsilon}$ is asymptotically normal.

On the other hand, as we shall see in Section \ref{S:Homogenization}, things get more complicated under Regime $1$ when  $b\neq 0$. In this case, the likelihood function
(\ref{Eq:LikelihoodFunction}) does not necessarily have a well defined limit due to the terms that are multiplied by $\epsilon/\delta$ (recall that in this case
$\epsilon/\delta\uparrow \infty$ as $\epsilon\downarrow 0$). We choose to resolve this issue, by taking the limit in an appropriately re-scaled and centered version of the original
 log-likelihood (a pseudo log-likelihood). Under certain conditions, this pseudo log-likelihood approach overcomes the convergence issue and a well defined limit exists. However,
the pseudo maximum likelihood estimator is consistent, even though the ``bias'' is explicitly characterized.

The consistency issue of  the maximum likelihood estimation in the presence of ``unbounded drift terms'',  such as the term
$\frac{\epsilon}{\delta}\int_{0}^{t}b\left(X^{\epsilon}_{s},\frac{X^{\epsilon}_{s}}{\delta}\right)ds$ with $\epsilon/\delta\uparrow\infty$ is well known in the literature.
In the context of $\epsilon=1$ and $\delta\downarrow 0$, which corresponds to Regime $1$, the problem has also been studied in  \cite{AzencottBeriTimofeyev2010a, AzencottBeriTimofeyev2010b,KrumscheidPavliotisKalliadasis2011, PaviotisStuart2007, PapavasileiouPaviotisStuart2009}
under different scenarios and conditions  and it is shown there that the maximum likelihood estimator is not consistent and one may need to result in sub-sampling
 of the data at appropriate rates in order to produce consistent estimators.   In the case that has been
 studied in \cite{AzencottBeriTimofeyev2010b,PapavasileiouPaviotisStuart2009} the issue was treated with appropriate sub-sampling of the data.
The article \cite{KrumscheidPavliotisKalliadasis2011} followed a semi-parametric approach assuming a special structure of the coefficients.
In this work we do not address the consistency issue.   Nevertheless,
 we provide an explicit formula for the asymptotic error in the transformed log-likelihood function. Moreover, we apply our results to the case of small noise diffusion in a two-scale
 potential field, see Section \ref{S:Examples}. In this case, even though, the original estimator is not consistent, we can construct
a consistent estimator and also derive a central limit theorem for the proposed estimator.


\section{Limiting Likelihood}\label{S:LimitngLikelihood}

We first study the limiting likelihood for Regime 1 when $b=0$ and for Regimes 2, 3 and then the proposed pseudo limiting likelihood for Regime 1 when $b\neq 0$.


\subsection{Limiting Likelihood for Regime $\mathbf{1}$ when $\mathbf{b=0}$ and for Regimes $\mathbf{2}$, $\mathbf{3}$.}\label{S:Averaging}
In this section, we  consider the limit of the likelihood function $Z_{\theta,T}^{\epsilon}(\mathcal{X}_{T})$, defined by (\ref{Eq:LikelihoodFunction}) for Regimes 1
when $b=0$ and for Regimes $2$ and $3$.

Let us define the following functions
\begin{definition}
\label{Def:ThreePossibleLikelihoods} For $z_{T}\doteq \left\{x_{t},0\leq t\leq T\right\}$, $x\in\mathbb{R}^{d}
,y\in\mathcal{Y}$ and for the three possible Regimes $i=1,2,3$
defined in (\ref{Def:ThreePossibleRegimes}), define
\begin{align}
\bar{Z}^{1}_{\theta,\theta_{0},T}(z_{\cdot})  &  =\int_{0}^{T}\int_{\mathcal{Y}}\left< c_{\theta},c_{\theta_{0}}\right>_{\alpha}\left(x_{s},y\right)\mu^{1}_{\theta_{0}}(dy;x_{s})ds
-\frac{1}{2}\int_{0}^{T}\int_{\mathcal{Y}}\left\Vert c_{\theta}\right\Vert^{2}_{\alpha}\left(x_{s},y\right)\mu^{1}_{\theta_{0}}(dy;x_{s})ds,\nonumber\\
\bar{Z}^{2}_{\theta,\theta_{0},T}(z_{\cdot})  &  =\int_{0}^{T}\int_{\mathcal{Y}}\left< \gamma b_{\theta}+c_{\theta},\gamma b_{\theta_{0}}+c_{\theta_{0}}\right>_{\alpha}\left(x_{s},y\right)\mu^{2}_{\theta_{0}}(dy;x_{s})ds
-\frac{1}{2}\int_{0}^{T}\int_{\mathcal{Y}}\left\Vert \gamma b_{\theta}+c_{\theta}\right\Vert^{2}_{\alpha}\left(x_{s},y\right)\mu^{2}_{\theta_{0}}(dy;x_{s})ds,\nonumber\\
\bar{Z}^{3}_{\theta,\theta_{0},T}(z_{\cdot})  &  =\int_{0}^{T}\int_{\mathcal{Y}}\left< c_{\theta},c_{\theta_{0}}\right>_{\alpha}\left(x_{s},y\right)\mu^{3}_{\theta_{0}}(dy;x_{s})ds
-\frac{1}{2}\int_{0}^{T}\int_{\mathcal{Y}}\left\Vert c_{\theta}\right\Vert^{2}_{\alpha}\left(x_{s},y\right)\mu^{3}_{\theta_{0}}(dy;x_{s})ds.\nonumber
\end{align}
\end{definition}
We then prove the following Theorem
\begin{theorem}\label{T:LikelihoodConvergence1}
Let the assumptions of Theorem \ref{T:LLN} hold. Let $\mathcal{X}_{T}=\left\{x_{t},0\leq t\leq T\right\}$ be a sample path of (\ref{Eq:LDPandA1}) at $\theta=\theta_{0}$. In the case of Regime $1$ we assume that $b_{\theta}=0$. Then, under Regime $i=1,2,3$, the sequence $\left\{Z_{\theta,T}^{\epsilon}, \epsilon>0\right\}$ converges in $\mathbb{P}_{\theta_{0}}$ probability, uniformly in $\theta\in\Theta$ to $\bar{Z}^{i}_{\theta,\theta_{0},T}\left(\bar{X}^{i}_{\cdot}\right)$ from Definition \ref{Def:ThreePossibleLikelihoods} and where $\bar{X}^{i}_{t}$ is the solution to the corresponding limiting ODE from Theorem \ref{T:LLN}. In particular, for any $\eta>0$
\begin{equation*}
\lim_{\epsilon\downarrow 0}\mathbb{P}_{\theta_{0}}\left[\sup_{\theta\in\Theta}\left|Z_{\theta,T}^{\epsilon}(\mathcal{X}_{T})-\bar{Z}^{i}_{\theta,\theta_{0},T}\left(\bar{X}^{i}_{\cdot}\right)\right|>\eta\right]=0.
\end{equation*}
Lastly, in each regime, the function $\bar{Z}^{i}_{\theta,\theta_{0},T}$ is maximized at $\theta=\theta_{0}$.
\end{theorem}
\begin{remark}
Here, $\bar{X}^{i}_{\cdot}=\bar{X}^{i}_{\cdot}(\theta_0)$, i.e., the parameter value is $\theta=\theta_{0}$. However, throughout this section we use the compact notation $\bar{X}^{i}_{\cdot}$ instead of $\bar{X}^{i}_{\cdot}(\theta_0)$ for presentation purposes only, in order to simplify the formulas.
\end{remark}

\begin{proof}
Since $\mathcal{X}_{T}=\left\{x_{t},0\leq t\leq T\right\}$ is a sample path of (\ref{Eq:LDPandA1}) at $\theta=\theta_{0}$, we get that
\begin{eqnarray}
Z_{\theta,T}^{\epsilon}(\mathcal{X}_{T})&=&
\int_{0}^{T}\left<\frac{\epsilon}{\delta}b_{\theta}+c_{\theta}, dx_{s}\right>_{\alpha}\left(x_{s},\frac{x_{s}}{\delta}\right)-
\frac{1}{2}\int_{0}^{T}\left\Vert \frac{\epsilon}{\delta}b_{\theta}+c_{\theta}\right\Vert^{2}_{\alpha}\left(x_{s},\frac{x_{s}}{\delta}\right) ds\nonumber\\
&=& I^{1,\epsilon}_{T}+\sqrt{\epsilon}I^{2,\epsilon}_{T},\nonumber
\end{eqnarray}
where
\begin{equation*}
I^{1,\epsilon}_{T}=\int_{0}^{T}\left<\frac{\epsilon}{\delta}b_{\theta}+c_{\theta}, \frac{\epsilon}{\delta}b_{\theta_{0}}+c_{\theta_{0}}\right>_{\alpha}\left(x_{s},\frac{x_{s}}{\delta}\right)ds-
\frac{1}{2}\int_{0}^{T}\left\Vert \frac{\epsilon}{\delta}b_{\theta}+c_{\theta}\right\Vert^{2}_{\alpha}\left(x_{s},\frac{x_{s}}{\delta}\right) ds
\end{equation*}
and
\begin{equation*}
I^{2,\epsilon}_{T}=\int_{0}^{T}\left<\frac{\epsilon}{\delta}b_{\theta}+c_{\theta}, \sigma dW_{s}\right>_{\alpha}\left(x_{s},\frac{x_{s}}{\delta}\right).
\end{equation*}
Then standard averaging principle for locally periodic diffusions, see Chapter 3 of \cite{BLP}, and  the fact that the corresponding invariant measure $\mu^{i}_{\theta}(dy;x)$ is continuous as a function of $x$ and Theorem \ref{T:LLN} imply that for any $p\geq 1$
\begin{equation*}
\mathbb{E}\left(I^{1,\epsilon}_{T}-\bar{Z}^{i}_{\theta,\theta_{0},T}\left(\bar{X}^{i}_{\cdot}\right)\right)^{p}\rightarrow 0, \text{ as }\epsilon\downarrow 0.
\end{equation*}
Moreover, the Burkholder-Davis-Gundy inequality \cite{KaratzasShreve} applied to the stochastic integral $\sqrt{\epsilon}I^{2,\epsilon}_{t}$ and Condition \ref{A:Assumption1} imply
\begin{equation*}
\mathbb{E}\sup_{0\leq t\leq T}|\sqrt{\epsilon}I^{2,\epsilon}_{t}|^{p}\leq C \epsilon^{p/2},
\end{equation*}
for some constant $C>0$, uniformly in $\theta\in\Theta$.

Thus, the proof of the claimed convergence follows by Chebyschev's inequality and the uniform convergence in $\theta\in\Theta$. The fact that the limit is maximized at $\theta=\theta_{0}$ is easily seen to hold by completing the square in the expressions for $\bar{Z}^{i}_{\theta,\theta_{0},T}$ at Definition \ref{Def:ThreePossibleLikelihoods}. For example, in the case of Regime 1, it is easy to see that
\begin{equation*}
\bar{Z}^{1}_{\theta,\theta_{0},T}(z_{\cdot})   =
\frac{1}{2}\int_{0}^{T}\int_{\mathcal{Y}}\left\Vert c_{\theta_{0}}\right\Vert^{2}_{\alpha}\left(x_{s},y\right)\mu^{1}_{\theta_{0}}(dy;x_{s})ds-
\frac{1}{2}\int_{0}^{T}\int_{\mathcal{Y}}\left\Vert c_{\theta}-c_{\theta_{0}}\right\Vert^{2}_{\alpha}\left(x_{s},y\right)\mu^{1}_{\theta_{0}}(dy;x_{s})ds
\end{equation*}
and thus the maximum is easily seen to be attained at $\theta=\theta_{0}$. Similarly for Regimes $2$ and $3$.
\end{proof}

Before we continue, we need to impose the following identifiability condition for the true value of the parameter $\theta$.
\begin{condition}\label{cond:identifiability}
For all $\eta>0$,
\begin{equation*}
\sup_{u: |u|>\eta} \left\{ \bar{Z}^{i}_{\theta_{0} + u,T}(\bar{X}^{i}_{\cdot}) - \bar{Z}^{i}_{\theta_{0},T}(\bar{X}^{i}_{\cdot}) \right\} \leq -\eta <0.\\
\end{equation*}
\end{condition}

\begin{theorem}
Let $\hat{\theta}^{\epsilon}\doteq\text{argmax}_{\theta\in\Theta} \;Z_{\theta,T}^{\epsilon}(\mathcal{X}_{T})$. Under Condition \ref{cond:identifiability} and the
assumptions of Theorem \ref{T:LikelihoodConvergence1}, the MLE sequence $\left\{\hat{\theta}^{\epsilon},\epsilon>0\right\}$  converges in $\mathbb{P}_{\theta_{0}}$ probability to the true parameter. In particular, for any $\eta>0$ we have
\begin{equation}
\lim_{\epsilon\downarrow 0}\mathbb{P}_{\theta_{0}}\left[\left|\hat{\theta}^{\epsilon}-\theta_{0}\right|>\eta\right]=0.
\end{equation}
\end{theorem}

\begin{proof}
For all $\eta>0$, we have that
\begin{align*}
& \mathbb{P}_{\theta_{0}} \left[ \left|\hat{\theta}^{\epsilon} -\theta_{0}\right| >\eta \right] \leq
\mathbb{P}_{\theta_{0}} \left[ \sup_{|u|>\eta} \left( Z_{u+\theta_{0}}^{\epsilon}(\mathcal{X}_{T}) - Z_{\theta_{0}}^{\epsilon}(\mathcal{X}_{T}) \right) \geq 0 \right]\\
 \leq  & \;\;\mathbb{P}_{\theta_{0}} \Biggl[ \sup_{|u|>\eta}  \left( \left( Z_{u+\theta_{0}}^{\epsilon}(\mathcal{X}_{T}) - Z_{\theta_{0}}^{\epsilon}(\mathcal{X}_{T}) \right) - \left( \bar{Z}_{u+\theta_{0}}^{i}\left(\bar{X}^{i}_{\cdot}\right) - \bar{Z}_{\theta_{0}}^{i}\left(\bar{X}^{i}_{\cdot}\right) \right) \right) \\
 & \qquad \qquad \qquad \qquad \qquad \qquad \qquad \qquad \qquad \geq -\sup_{|u|>\eta} \left(\bar{Z}^{i}_{\theta_{0}+u} \left(\bar{X}^{i}_{\cdot}\right) - \bar{Z}^{i}_{\theta_{0}}\left(\bar{X}^{i}_{\cdot}\right)\right) \Biggr].
\end{align*}
Condition \ref{cond:identifiability} gives that
\begin{align*}
& \mathbb{P}_{\theta_{0}} \Biggl[ \sup_{|u|>\eta}  \left( \left( Z_{u+\theta_{0}}^{\epsilon}(\mathcal{X}_{T}) - Z_{\theta_{0}}^{\epsilon}(\mathcal{X}_{T}) \right) - \left( \bar{Z}_{u+\theta_{0}}^{i} \left(\bar{X}^{i}_{\cdot}\right)- \bar{Z}_{\theta_{0}}^{i} \left(\bar{X}^{i}_{\cdot}\right)\right) \right) \\
&  \qquad \qquad \qquad \qquad \quad \qquad \qquad \quad \quad \quad \qquad \geq -\sup_{|u|>\eta} \left(\bar{Z}^{i}_{u+\theta_{0}} \left(\bar{X}^{i}_{\cdot}\right) - \bar{Z}^{i}_{\theta_{0}}\left(\bar{X}^{i}_{\cdot}\right)\right) \Biggr] \\
\leq & \;\;\mathbb{P}_{\theta_{0}} \left[ \sup_{|u|>\eta}  \left( \left( Z_{u+\theta_{0}}^{\epsilon}(\mathcal{X}_{T}) - \bar{Z}_{u+\theta_{0}}^{i}\left(\bar{X}^{i}_{\cdot}\right)  \right) - \left( Z_{\theta_{0}}^{\epsilon}(\mathcal{X}_{T}) - \bar{Z}_{\theta_{0}}^{i}\left(\bar{X}^{i}_{\cdot}\right)  \right) \right) \geq \eta >0 \right].
\end{align*}
Therefore, by conditioning on $\left\{\left| Z_{\theta_{0}}^{\epsilon}(\mathcal{X}_{T}) - \bar{Z}_{\theta_{0}}^{i}\left(\bar{X}^{i}_{\cdot}\right)\right| \geq \frac{1}{2}\eta \right\}$ we have
\begin{equation*}
\mathbb{P}_{\theta_{0}} \left[ \left|\hat{\theta}^{\epsilon} -\theta_{0}\right| >\eta \right] \leq  \;\;
\mathbb{P}_{\theta_{0}} \left[\sup_{|u|>\eta} \left( Z_{u+\theta_{0}}^{\epsilon} - \bar{Z}_{u+\theta_{0}}^{i}\right) \geq \frac{1}{2}\eta   >0 \right] + \mathbb{P}_{\theta_{0}} \left[ \left| Z_{\theta_{0}}^{\epsilon} - \bar{Z}_{\theta_{0}}^{i}\right| \geq \frac{1}{2}\eta >0 \right].
\end{equation*}
The result follows by the uniform convergence of Theorem \ref{T:LikelihoodConvergence1}.
\end{proof}


\subsection{Pseudo Limiting Likelihood for Regime $\mathbf{1}$ when $\mathbf{b \neq0}$.}\label{S:Homogenization}

In the case of Regime $1$ with $b\neq 0$, the situation is more involved because the limit of the log-likelihood $Z_{\theta,T}^{\epsilon}(\mathcal{X}_{T})$ by (\ref{Eq:LikelihoodFunction}) is not well defined. This is due to the $\epsilon/\delta$ and $(\epsilon/\delta)^{2}$ terms that appear in the expression of $Z_{\theta,T}^{\epsilon}(\mathcal{X}_{T})$. This leads us to re-parameterize the log-likelihood, so that it will have a well defined limit. However, we need to re-parameterize the log-likelihood in such a way so that the limiting expression will coincide with the expression of Section \ref{S:Averaging} for $b=0$ and at the same time maintain tractability and simplicity.

Let us denote by $Z_{\theta,T}^{\epsilon}(\mathcal{X}_{T};0)$ the log-likelihood function (\ref{Eq:LikelihoodFunction}) with $b=0$. We define the modified log-likelihood function
\begin{equation}
\hat{Z}_{\theta,T}^{\epsilon}(\mathcal{X}_{T})=\left(\frac{\delta}{\epsilon}\right)^{2}Z_{\theta,T}^{\epsilon}(\mathcal{X}_{T})+Z_{\theta,T}^{\epsilon}(\mathcal{X}_{T};0).\label{Eq:Regime1b}
\end{equation}
To characterize the limit, we first need to define several quantities.  But first we impose an additional assumption.

\begin{condition}\label{A:Assumption4}
Let the coefficients $b_{\theta},c_{\theta}$ and $\sigma$ be such that
\[
\int_{\mathcal{Y}}\left<b_{\theta_{0}},c_{\theta}\right>_{\alpha}(x,y)\mu^{1}_{\theta_{0}}(dy;x)=0
\]
for all $\theta,\theta_{0}\in\Theta$ and for all $x\in\mathbb{R}^{d}$.
\end{condition}

For example, Condition \ref{A:Assumption4} is trivially satisfied under Condition \ref{A:Assumption2} if the coefficients $c_{\theta}$ and $\sigma$ are independent of $y\in\mathcal{Y}$
(see also Remark \ref{R:RemarkRegime1} below).

Then, we can consider the auxiliary partial differential equation
\begin{equation}
\mathcal{L}^{1}_{x}\Phi(x,y)=-\left<b_{\theta_{0}},c_{\theta}\right>_{\alpha}(x,y), \qquad \int_{\mathcal{Y}}\Phi(x,y)\mu^{1}_{\theta_{0}}(dy;x)=0.\label{Eq:PoissonEquation}
\end{equation}
Under Condition \ref{A:Assumption4}, this Poisson equation has a unique bounded, periodic in $y$ and smooth solution (see Theorem 3.3.4 of \cite{BLP}). In order to emphasize the dependence of $\Phi$ on $\theta,\theta_{0}$, we shall often write  $\Phi_{\theta,\theta_{0}}(x,y)$.

Next, we define
\begin{eqnarray}
J^{1}_{\theta,\theta_{0},T}(z_{\cdot})&=&\int_{0}^{T}\int_{\mathcal{Y}}\left< b_{\theta},b_{\theta_{0}}\right>_{\alpha}\left(x_{s},y\right)\mu^{1}_{\theta_{0}}(dy;x_{s})ds
-\frac{1}{2}\int_{0}^{T}\int_{\mathcal{Y}}\left\Vert b_{\theta}\right\Vert^{2}_{\alpha}\left(x_{s},y\right)\mu^{1}_{\theta_{0}}(dy;x_{s})ds\nonumber\\
& &+\int_{0}^{T}\int_{\mathcal{Y}}\left< c_{\theta},c_{\theta_{0}}\right>_{\alpha}\left(x_{s},y\right)\mu^{1}_{\theta_{0}}(dy;x_{s})ds
-\frac{1}{2}\int_{0}^{T}\int_{\mathcal{Y}}\left\Vert c_{\theta}\right\Vert^{2}_{\alpha}\left(x_{s},y\right)\mu^{1}_{\theta_{0}}(dy;x_{s})ds\label{Eq:MainTermRegime1general}
\end{eqnarray}
and
\begin{equation*}
H_{\theta,\theta_{0}}(z_{\cdot})=\int_{0}^{T}\int_{\mathcal{Y}} \left<c_{\theta_{0}}(x_{s},y), \nabla_{y}\Phi_{\theta,\theta_{0}}\left(x_{s},y\right)\right>\mu^{1}_{\theta_{0}}(dy;x_{s})ds.
\end{equation*}

For the limiting distribution we prove the following theorem
\begin{theorem}\label{T:LikelihoodConvergence2}
Let Conditions \ref{A:Assumption1}, \ref{A:Assumption2} and \ref{A:Assumption4} hold and consider Regime $1$. Let $\mathcal{X}_{T}=\left\{x_{t},0\leq t\leq T\right\}$ be a sample path of (\ref{Eq:LDPandA1}) at $\theta=\theta_{0}$. Then, the sequence $\left\{\hat{Z}_{\theta,T}^{\epsilon}, \epsilon>0\right\}$, as defined by (\ref{Eq:Regime1b}), converges in $\mathbb{P}_{\theta_{0}}$ probability, uniformly in $\theta\in\Theta$ to $\hat{Z}^{1}_{\theta,\theta_{0},T}\left(\bar{X}^{1}_{\cdot}\right)$, where
\begin{equation*}
\hat{Z}^{1}_{\theta,\theta_{0},T}(z_{\cdot})=J^{1}_{\theta,\theta_{0}}(z_{\cdot})+H_{\theta,\theta_{0}}(z_{\cdot}).
\end{equation*}
In particular, for any $\eta>0$
\begin{equation*}
\lim_{\epsilon\downarrow 0}\mathbb{P}_{\theta_{0}}\left[\sup_{\theta\in\Theta}\left|\hat{Z}_{\theta,T}^{\epsilon}(\mathcal{X}_{T})-\hat{Z}^{1}_{\theta,\theta_{0},T}\left(\bar{X}^{1}_{\cdot}\right)\right|>\eta\right]=0.
\end{equation*}
\end{theorem}
Before proceeding with the proof of the theorem, we make two remarks.
\begin{remark}
 When $b=0$ we get that the ``bias'' $H_{\theta,\theta_{0}}(z_{\cdot})=0$ (since in this case $\Phi(x,y)=0$), and we get back the result of Theorem \ref{T:LikelihoodConvergence1}.
The term $J^{1}_{\theta,\theta_{0}}(z_{\cdot})$ is maximized at $\theta=\theta_{0}$ as in Theorem \ref{T:LikelihoodConvergence1}.
However, this is not true in general for $H_{\theta,\theta_{0}}(z_{\cdot})$. This implies that
 maximum likelihood in general fails for Regime $1$.
\end{remark}

\begin{remark}\label{R:RemarkRegime1}
When Condition \ref{A:Assumption4} is not satisfied, the situation is more complicated. Using the modified log-likelihood (\ref{Eq:Regime1b}), Condition \ref{A:Assumption4} is necessary in order for (\ref{Eq:PoissonEquation})
to have a solution. This follows by Fredholm alternative as in Theorem 3.3.4 of \cite{BLP}. The use of the Poisson equation (\ref{Eq:PoissonEquation}) is an essential tool in the proof of Theorem \ref{T:LikelihoodConvergence2}.
There does not seem to be an obvious way to reparameterize
 the likelihood in such a way that it will have a well defined limit and at the same time maintain tractability.
However, as we shall see in Section \ref{S:Examples}, Theorem \ref{T:LikelihoodConvergence2} covers one of the cases of interest which is the first order Langevin equation with a two scale potential.
 To be more precise, it covers the case of a small noise diffusion in two-scale potentials of the form (\ref{Eq:LDPandA1}) with $b_{\theta}(x,y)=-\nabla Q_{\theta}(y)$, $c_{\theta}(x,y)=-\nabla V_{\theta}(x)$
and $\sigma(x,y)=$constant.
\end{remark}

\begin{proof}[Proof of Theorem \ref{T:LikelihoodConvergence2}.]
After some term rearrangement, we get
\begin{eqnarray}
\hat{Z}_{\theta,T}^{\epsilon}(\mathcal{X}_{T})&=&\int_{0}^{T}\left[\left<b_{\theta},b_{\theta_{0}}\right>_{\alpha}-\frac{1}{2}\left\Vert b_{\theta}\right\Vert_{\alpha}^{2}\right]\left(x_{s},\frac{x_{s}}{\delta}\right)ds\nonumber\\
& &\hspace{0.2cm}+\frac{\epsilon}{\delta}\int_{0}^{T}\left<b_{\theta_{0}},c_{\theta}\right>_{\alpha}\left(x_{s},\frac{x_{s}}{\delta}\right)ds\nonumber\\
& &\hspace{0.2cm}+\int_{0}^{T}\left[\left<c_{\theta},c_{\theta_{0}}\right>_{\alpha}-\frac{1}{2}\left\Vert c_{\theta}\right\Vert_{\alpha}^{2}\right]\left(x_{s},\frac{x_{s}}{\delta}\right)ds+\nonumber\\
& &\hspace{0.2cm}+\frac{\delta}{\epsilon}\int_{0}^{T}\left[\left<b_{\theta},c_{\theta_{0}}\right>_{\alpha}+\left<b_{\theta_{0}},c_{\theta}\right>_{\alpha}-\left<b_{\theta},c_{\theta}\right>_{\alpha}\right]\left(x_{s},\frac{x_{s}}{\delta}\right)ds\nonumber\\
& &\hspace{0.2cm}+\left(\frac{\delta}{\epsilon}\right)^{2}\int_{0}^{T}\left[\left<c_{\theta},c_{\theta_{0}}\right>_{\alpha}-\frac{1}{2}\left\Vert c_{\theta}\right\Vert_{\alpha}^{2}\right]\left(x_{s},\frac{x_{s}}{\delta}\right)ds+
\nonumber\\
& &+\hspace{0.2cm}\sqrt{\epsilon}\left[\frac{\delta}{\epsilon}\int_{0}^{T}\left< b_{\theta},\sigma dW_{s}\right>_{\alpha}\left(x_{s},\frac{x_{s}}{\delta}\right)+\left(\left(\frac{\delta}{\epsilon}\right)^{2}+1\right)\int_{0}^{T}\left< c_{\theta},\sigma dW_{s}\right>_{\alpha}\left(x_{s},\frac{x_{s}}{\delta}\right)\right]\nonumber\\
&=&K_{1}^{\epsilon}+
\frac{\epsilon}{\delta}K_{2}^{\epsilon} +K_{3}^{\epsilon}+\frac{\delta}{\epsilon}K_{4}^{\epsilon}+\left(\frac{\delta}{\epsilon}\right)^{2}K_{5}^{\epsilon}+\sqrt{\epsilon}M^{\epsilon}_{T}. \label{Eq:Likelihood1}
\end{eqnarray}

We study the limiting behavior of the terms in the right hand side of (\ref{Eq:Likelihood1}).
It is relatively easy to see that the $\frac{\delta}{\epsilon}K_{4}^{\epsilon}+\left(\frac{\delta}{\epsilon}\right)^{2}K_{5}^{\epsilon}$ converges to zero in the p-th mean for every $p\geq 1$. Moreover, the quadratic variation of the stochastic integral $M^{\epsilon}_{T}$ in (\ref{Eq:Likelihood1}) has a well defined limit in p-th mean, which together with the fact that it is multiplied by $\sqrt{\epsilon}$, gives us that this term on the right hand side of
(\ref{Eq:Likelihood1}) converges to zero in p-th mean.

Therefore it remains to study the terms $K_{1}^{\epsilon}$, $\frac{\epsilon}{\delta}K_{2}^{\epsilon}$ and $K_{3}^{\epsilon}$. By standard averaging principle for locally periodic diffusions, it can be seen that $K_{1}^{\epsilon}+K_{3}^{\epsilon}$ converges in $\mathbb{P}_{\theta_{0}}$ probability, uniformly in $\theta\in\Theta$ to $J^{1}_{\theta,\theta_{0}}(\bar{X}^{1}_{\cdot})$; see for example \cite{BLP,PS}.

Lastly, we need to study the term $\frac{\epsilon}{\delta}K_{2}^{\epsilon}$. For this purpose we apply It\^{o} formula to $\Phi(x,x/\delta)$ that satisfies (\ref{Eq:PoissonEquation}) with $x=X^{\epsilon}_{s}$ to get
\begin{eqnarray}
d\Phi&=&\left[\frac{\epsilon}{\delta^{2}}\mathcal{L}^{1}_{X_{t}}\Phi+\frac{1}{\delta}\left<c_{\theta_{0}},\nabla_{y}\Phi\right>\right]dt\nonumber\\
& &+\left[\frac{\epsilon}{\delta}\left<b_{\theta_{0}},\nabla_{x}\Phi \right>+\left<c_{\theta_{0}},\nabla_{x}\Phi\right>+ \frac{\epsilon}{2}\sigma\sigma^{T}:\nabla_{x}\nabla_{x}\Phi+\frac{\epsilon}{\delta}\sigma\sigma^{T}:\nabla_{x}\nabla_{y}\Phi\right]dt\nonumber\\
& &+\frac{\sqrt{\epsilon}}{\delta}\left<\nabla_{y}\Phi,\sigma dW_{t}\right>+\sqrt{\epsilon}\left<\nabla_{x}\Phi,\sigma dW_{t}\right>.\label{Eq:ItoFormula1}
\end{eqnarray}
Hence, recalling that $\Phi$ satisfies (\ref{Eq:PoissonEquation}), which has a unique, periodic in $y$, bounded and smooth solution due to Condition \ref{A:Assumption4}, we obtain
\begin{eqnarray}
\frac{\epsilon}{\delta}K_{2}^{\epsilon}&=&\frac{\epsilon}{\delta}\int_{0}^{T}\left<b_{\theta_{0}},c_{\theta}\right>_{\alpha}\left(X^{\epsilon}_{s},\frac{X^{\epsilon}_{s}}{\delta}\right)ds=-
\frac{\epsilon}{\delta}\int_{0}^{T}\mathcal{L}^{1}_{X_{s}}\Phi \left(X_{s}^{\epsilon},\frac{X_{s}^{\epsilon}}{\delta}\right)ds\nonumber\\
&=& \delta\left(\Phi(0)-\Phi(t)\right)\nonumber\\
& &+\int_{0}^{T}\left[\epsilon\left<b_{\theta_{0}},\nabla_{x}\Phi \right>+\delta\left<c_{\theta_{0}},\nabla_{x}\Phi\right>+ \frac{\epsilon\delta}{2}\sigma\sigma^{T}:\nabla_{x}\nabla_{x}\Phi+\epsilon\sigma\sigma^{T}:\nabla_{x}\nabla_{y}\Phi\right]\left(X_{s}^{\epsilon},\frac{X_{s}^{\epsilon}}{\delta}\right)ds\nonumber\\
& &+ \sqrt{\epsilon}\int_{0}^{T}\left<\nabla_{y}\Phi,\sigma dW_{s}\right>\left(X_{s}^{\epsilon},\frac{X_{s}^{\epsilon}}{\delta}\right)+\sqrt{\epsilon}\delta\int_{0}^{T}\left<\nabla_{x}\Phi,\sigma dW_{s}\right>\left(X_{s}^{\epsilon},\frac{X_{s}^{\epsilon}}{\delta}\right)\nonumber\\
& &+\int_{0}^{T}\left<c_{\theta_{0}},\nabla_{y}\Phi\right>\left(X_{s}^{\epsilon},\frac{X_{s}^{\epsilon}}{\delta}\right) ds.\label{Eq:ItoFormula2}
\end{eqnarray}

From this statement the result follows immediately since the last term $\int_{0}^{T}\left<c_{\theta_{0}},\nabla_{y}\Phi\right>\left(X_{s}^{\epsilon},\frac{X_{s}^{\epsilon}}{\delta}\right) ds$ converges in $\mathbb{P}_{\theta_{0}}$ probability, uniformly in $\theta\in\Theta$, to $H_{\theta,\theta_{0}}(\bar{X}^{1}_{\cdot})$. The rest of the terms on the right hand side of the last display converge to zero in $\mathbb{P}_{\theta_{0}}$ probability, uniformly in $\theta\in\Theta$, due to the boundedness of $\Phi$ and its derivatives and Condition \ref{A:Assumption1}. This concludes the proof of the theorem.
\end{proof}


\section{Central Limit Theorem for Regime 1 when $b=0$ and for Regimes 2 and 3}\label{S:CLT}

In this section we state and prove a central limit theorem (CLT) for the maximum likelihood estimator $\hat{\theta}^{\epsilon}$ of $\theta$ in the case of Subsection \ref{S:Averaging}.

The main structural assumption is that under Regime 1 we have that $b_{\theta}(x,y)=0$.
For notational convenience and without loss of generality, we then consider that $b_{\theta}(x,y)=0$ for all three regimes. For Regime 2 when $b_{\theta}\neq 0$ one essentially  just replaces in the final formula, the function $c_{\theta}(x,y)$, by the function $\gamma b_{\theta}(x,y)+c_{\theta}(x,y)$.
So, without loss of generality, let us assume that $b(x,y)=0$ for all three regimes.

We define the normed log-likelihood ratio
\begin{eqnarray}
M_{\epsilon}(\theta,u)&=&\log\frac{d\mathbb{P}_{\theta+\sqrt{\epsilon} u}}{d\mathbb{P}_{\theta}}(x)\nonumber\\
&=&
\frac{1}{\epsilon}\int_{0}^{T}\left<c_{\theta+\sqrt{\epsilon}u}-c_{\theta}, dx_{s}\right>_{\alpha}\left(x_{s},\frac{x_{s}}{\delta}\right)-
\frac{1}{2\epsilon}\int_{0}^{T}\left(\left\Vert c_{\theta+\sqrt{\epsilon}u}\right\Vert^{2}_{\alpha}-\left\Vert c_{\theta}\right\Vert^{2}_{\alpha}\right)\left(x_{s},\frac{x_{s}}{\delta}\right) ds.\nonumber
\end{eqnarray}
With $\mathbb{P}_{\theta}$ probability $1$ we can write that
\begin{equation*}
M_{\epsilon}(\theta,u)=\frac{1}{\sqrt{\epsilon}}\int_{0}^{T}
\left< c_{\theta+\sqrt{\epsilon}u}-c_{\theta}, \sigma dW_{s}\right>_{\alpha}\left(x_{s},\frac{x_{s}}{\delta}\right)-
\frac{1}{2\epsilon}\int_{0}^{T}\left\Vert c_{\theta+\sqrt{\epsilon}u}-c_{\theta}\right\Vert^{2}_{\alpha}\left(x_{s},\frac{x_{s}}{\delta}\right) ds.
\end{equation*}
For notational convenience, we also define the quantities
\begin{eqnarray}
S(\theta,x,y)&=&\sigma^{-1}(x,y)\nabla_{\theta}c_{\theta}(x,y),\nonumber\\
q_{i}(x,\theta)&=&\int_{\mathcal{Y}}S(\theta,x,y)S^{T}(\theta,x,y)\mu^{i}_{\theta}(dy;x).\label{Eq:Definition_q}
\end{eqnarray}
The Fisher information matrix is defined to be
\begin{equation*}
I_{i}(\theta)=\int_{0}^{T}q_{i}(\bar{X}^{i}_{s},\theta)ds.\label{Eq:FisherInformationMatrix}
\end{equation*}
To this end we recall that the invariant measure $\mu$ has, under our assumptions, a smooth, uniformly bounded away from zero density, which is also periodic in the $y-$variable (Theorem 3.3.4 and Section 3.6.2 in \cite{BLP}) for 
Regimes 1 and 2 and Condition 2.3 for Regime 3. We will denote by $m^{i}_{\theta}(x,y)$ the density of $\mu^{i}_{\theta}(dy;x)$, namely $\mu^{i}_{\theta}(dy;x)=m^{i}_{\theta}(x,y)dy$. In this section the following condition is imposed.
\begin{condition}\label{A:AssumptionCLT}
\begin{enumerate}
\item{The function $c_{\theta}(x,y)$ is twice continuously differentiable in $\theta$ with bounded derivatives.}
\item{The Fisher information matrix $I_{i}(\theta)$ is positive definite uniformly in $\theta\in \Theta$, i.e. there exists $c_{0}>0$ such that
\[
0<c_{0}\leq\inf_{\theta\in\Theta}\inf_{|\lambda|=1}\left<I(\theta)\lambda,\lambda\right>
\]}
\item{The vector process $\left\{q^{1/2}_{i}(X^{\epsilon}_{s},\theta), t\in[0,T]\right\}$ is continuous in probability, uniformly on $\theta\in\Theta$ in $L^{2}[0,T]$ on $\theta$ and on $X$ in the point
$\theta=\theta_{0}$.}
\item{The vector valued function $\sqrt{m^{i}_{\theta}(x,y)}\sigma^{-1}(x,y)c_{\theta}(x,y)$ is Lipschitz continuous in $x$ with a Lipschitz constant that is uniformly bounded in $(\theta,y)\in\Theta\times \mathcal{Y}$.}
\end{enumerate}
\end{condition}
We then have the following theorem.
\begin{theorem}\label{T:CLT}
Let the conditions of Theorem \ref{T:LLN} and Condition \ref{A:AssumptionCLT} hold. Consider Regime $i=1,2,3$ and let $\hat{\theta}^{\epsilon}$ be the maximum likelihood estimator of $\theta$.  Then, uniformly on compacts $\tilde{\Theta}\subset\Theta$ we have that in distribution under $\mathbb{P}_{\theta}$, the following central limit result holds
\begin{equation*}
\frac{1}{\sqrt{\epsilon}}\left[I_{i}(\theta)\right]^{1/2}\left(\hat{\theta}^{\epsilon}-\theta\right)\Rightarrow N(0,I).
\end{equation*}
Moreover, the MLE has converging moments for all $p>0$, i.e.,
\begin{equation*}
 \lim_{\epsilon\downarrow 0}\sup_{\theta\in\tilde{\Theta}}\left|\mathbb{E}_{\theta}\left|I_{i}^{1/2}(\theta)\left(\hat{\theta}^{\epsilon}-\theta\right)\right|^{p}\epsilon^{-p/2}-\mathbb{E}|Z|^{p}\right|=0
\end{equation*}
where $Z$ is a standard $N(0,I)$ random vector.
\end{theorem}

The proof of this theorem follows by Theorem 1.6 in Chapter $1$ of \cite{Kutoyants1994}). The Lemmas \ref{L:CLT_Condition1}, \ref{L:CLT_Condition2} and \ref{L:CLT_Condition3} prove that the conditions of that theorem hold. For notational convenience we omit writing the subscript $i$, which denotes the particular regime under consideration.

\begin{lemma}\label{L:CLT_Condition1}
Under the conditions of Theorem \ref{T:CLT}, the family
$\{\mathbb{P}^{\epsilon}_{\theta}: \theta\in \Theta\}$ is uniformly asymptotically normal with normalizing matrix $\phi(\epsilon,\theta)=\sqrt{\epsilon}I^{-1/2}(\theta)$.
\end{lemma}
The proof of this lemma is presented in the appendix.

\begin{lemma}\label{L:CLT_Condition2}
Under the conditions of Theorem \ref{T:CLT}, there exists $m>d/2$ and a constant $C<\infty$ such that for every $\epsilon\in(0,1)$ and compact $\tilde{\Theta}\subset \Theta$
\[
\sup_{\theta\in\tilde{\Theta}}\sup_{|u_{1},u_{2}|<r}|u_{2}-u_{1}|^{-2m}\mathbb{E}_{\theta}\left|e^{\frac{1}{2m}M_{\epsilon}(\theta,u_{2})}-e^{\frac{1}{2m}M_{\epsilon}(\theta,u_{1})} \right|^{2m}\leq C
\]
\end{lemma}
The proof of this lemma is presented in the appendix.

\begin{lemma}\label{L:CLT_Condition3}
Under the conditions of Theorem \ref{T:CLT}, and for any $p\in(0,1)$ and compact $\tilde{\Theta}\subset \Theta$ there exists a function $g_{\tilde{\Theta},p}(\left\Vert u \right\Vert)$ with the property
\[
\lim_{u\rightarrow\infty}u^{n}e^{-g_{\tilde{\Theta},p}(\left\Vert u\right\Vert)}=0,\quad \forall n\in\mathbb{N}
\]
such that
\[
\sup_{\theta\in\tilde{\Theta}}\mathbb{E}_{\theta}e^{p M_{\epsilon}(\theta,u)}\leq e^{-g_{\tilde{\Theta},p}(\left\Vert u\right\Vert)}
\]
\end{lemma}
The proof of this lemma is presented in the appendix.


\section{First Order Langevin Equation}\label{S:Examples}

A particular model of interest is the first order Langevin equation
\begin{equation*}
dX^{\epsilon}_{t}=-\nabla V^{\epsilon}_{\theta}\left(  X^{\epsilon}_{t},\frac
{X^{\epsilon}_{t}}{\delta}\right)  dt+\sqrt{\epsilon}\sqrt{2D}dW_t,\hspace
{0.2cm}X^{\epsilon}_{0}=x_{0},
\end{equation*}
where $V^{\epsilon}$ is some potential function and $2D$ the diffusion
constant. We are particularly interested in the case where the potential
function $V^{\epsilon}$ is composed of a large-scale smooth part and a fast
oscillating part of smaller magnitude:
\begin{equation}
V^{\epsilon}_{\theta}\left(  x,x/\delta\right)  =\epsilon Q(x/\delta)+\theta V(x).
\end{equation}
Thus the equation of interest can be written as
\begin{equation}
dX_{t}^{\epsilon}=\left[  -\frac{\epsilon}{\delta}\nabla Q\left(  \frac
{X_{t}^{\epsilon}}{\delta}\right)  -\theta\nabla V\left(  X_{t}^{\epsilon}\right)
\right]  dt+\sqrt{\epsilon}\sqrt{2D}dW_{t},\hspace{0.2cm}X_{0}^{\epsilon
}=x_{0}, \label{Eq:LangevinEquation2}
\end{equation}
An example of such a potential is given in Figure 1.
\begin{figure}[!h]
\label{F:Figure1}
\par
\begin{center}
\includegraphics[scale=0.4, width=6 cm, height=8 cm, angle=-90]{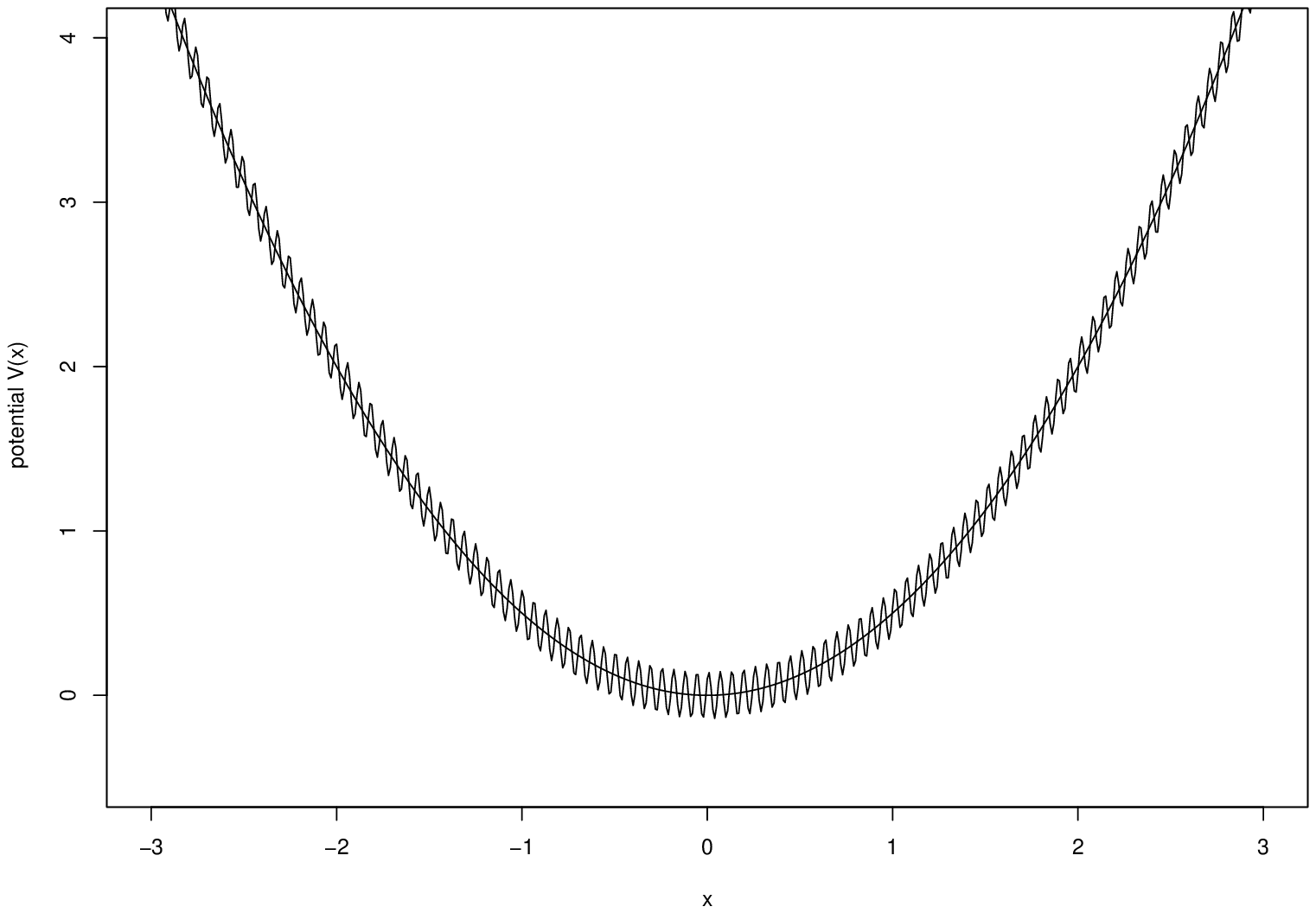}
\end{center}
\caption{ $V^{\epsilon}(x,\frac{x}{\delta})=V_{\theta}(x)+\epsilon Q(\frac{x}{\delta})$ with $V_{\theta}(x)=\frac{\theta}{2}x^{2}$,  $Q(\frac{x}{\delta})=  \cos(\frac
{x}{\delta})+\sin(\frac{x}{\delta})$ and parameters
 $\epsilon=0.1$, $\delta=0.01$ and $\theta=1$.}%
\end{figure}

For the potential function drawn in Figure \ref{F:Figure1}, the unkown parameter $\theta$ corresponds to the curvature of $V(x)$ around the equilibrium point.

We are interested in the statistical estimation problem for the parameter $\theta$ in the case of Regime $1$, i.e., when $\epsilon/\delta\uparrow\infty$. In Subsection \ref{SS:LangevinEquationMLE} we study the estimation problem for $\theta$ based on the methodology described in Subsection \ref{S:Homogenization}. In Subsection \ref{SS:LangevinEquationCLT} we study the corresponding central limit theorem. In Subsection \ref{SS:LangevinEquationSimulation} we present a simulation study.


\subsection{Pseudo Limiting Likelihood and Proposed Estimator for $\theta$} \label{SS:LangevinEquationMLE}
To connect to our notation let $b_{\theta}(x,y)=-\nabla Q(y)$, $c_{\theta}(x,y)=-\theta\nabla V(x)$, $\sigma(x,y)=\sqrt{2D}I$ and we consider Regime 1. In this case there is an explicit formula
for the invariant density $\mu(y)$, which is the Gibbs distribution
\begin{equation*}
\mu(dy)=\frac{1}{Z}e^{-\frac{Q(y)}{D}}dy,\hspace{0.2cm}Z=\int_{\mathcal{Y}%
}e^{-\frac{Q(y)}{D}}dy.
\end{equation*}
Moreover, it is easy to see that the centering Conditions \ref{A:Assumption2} and \ref{A:Assumption4}
hold. Notice that in this case the invariant measure does not depend neither on $x\in\mathbb{R}^{d}$, nor on $\theta\in\mathbb{R}$.
We also define
\begin{equation*}
\hat{Z}=\int_{\mathcal{Y}}e^{\frac{Q(y)}{D}}dy.
\end{equation*}
We have the following proposition.
\begin{proposition}\label{P:Langevin}
Under the conditions and notation of Theorem \ref{T:LikelihoodConvergence2} we have that the error term is given by
\begin{equation}
H_{\theta,\theta_{0}}(z_{\cdot})=\frac{\theta\theta_{0}}{2D}\int_{0}^{T}\left<\nabla V(x_{s}), \left(\int_{\mathcal{Y}} \frac{\partial \chi\left(y\right)}{\partial y}\mu(dy)\right)\nabla V(x_{s})\right>ds.\label{Eq:ErrorLangevinGeneral}
\end{equation}
\end{proposition}

\begin{proof}
In the case $b_{\theta}(x,y)=-\nabla Q(y)$, $c_{\theta}(x,y)=-\theta\nabla V(x)$ and $\sigma(x,y)=\sqrt{2D}I$, we notice that the solution $\Phi$ to the Poisson equation (\ref{Eq:PoissonEquation}) is related to the solution of the cell problem $\chi$, (\ref{Eq:CellProblem}), via the relation
\begin{equation*}
\Phi_{\theta}(x,y)=-\theta \frac{1}{2D}\left<\chi(y),\nabla V(x)\right>.
\end{equation*}
Hence, we have that

\begin{eqnarray}
H_{\theta,\theta_{0}}(z_{\cdot})&=&\int_{0}^{T}\left<c_{\theta_{0}}(x_{s}), \int_{\mathcal{Y}} \nabla_{y}\Phi_{\theta}\left(x_{s},y\right)\mu(dy)\right>ds\nonumber\\
&=&\frac{\theta\theta_{0}}{2D}\int_{0}^{T}\left<\nabla V(x_{s}), \left(\int_{\mathcal{Y}} \frac{\partial \chi\left(y\right)}{\partial y}\mu(dy)\right)\nabla V(x_{s})\right>ds.\nonumber
\end{eqnarray}
This concludes the proof of the proposition.
\end{proof}

When we have a separable fluctuating part, i.e. $Q(y_{1},y_{2},\ldots
,y_{d})=Q_{1}(y_{1})+Q_{2}(y_{2})+\cdots+Q_{d}(y_{d})$, everything can be
calculated explicitly. We summarize the results in the following corollary. This corollary also shows that in this case we can derive a consistent estimator $\theta_{0}$ in closed form.

\begin{corollary}\label{C:corol}
Assume $Q(y_{1},y_{2}%
,\cdots,y_{d})=Q_{1}(y_{1})+Q_{2}(y_{2})+\cdots+Q_{d}(y_{d})$ and consider
Regime $1$. Under the conditions and notation of Theorem \ref{T:LikelihoodConvergence2}, we have that the error term is given by
\begin{equation*}
H_{\theta,\theta_{0}}(z_{\cdot})=\frac{\theta\theta_{0}}{2D}\int_{0}^{T}\left<\nabla V(x_{s}), \left(-I+\lambda^{2}\Gamma\right)\nabla V(x_{s})\right>ds,
\end{equation*}
where $\lambda>0$ is the (common) period of the functions $Q_{i}$ in the corresponding direction,
\begin{equation*}
\Gamma=\textrm{diag}\left[\frac{1}{Z_{1}\hat{Z}_{1}},\cdots,\frac{1}{Z_{d}\hat{Z}_{d}}\right]
\end{equation*}
and for $i=1,2,\ldots,d$
\begin{equation*}
Z_{i}=\int_{\mathbb{T}}e^{-\frac{Q_{i}(y_{i})}{D}}dy_{i},\hspace{0.2cm}\hat
{Z}_{i}=\int_{\mathbb{T}}e^{\frac{Q_{i}(y_{i})}{D}}dy_{i}.
\end{equation*}
Moreover, we have that $H_{\theta,\theta_{0}}(z_{\cdot})\leq 0$.

Furthermore, recall the MLE $\hat{\theta}_{\epsilon}$. If $\int_{0}^{T}\left\Vert \nabla V(x_{s})\right\Vert^{2}_{K}ds\neq 0$ for both $K=I$ and $K=\Gamma^{-1}$, then
\begin{equation*}
 \tilde{\theta}^{\epsilon}=\left(\lambda^{2}\frac{\int_{0}^{T}\left\Vert \nabla V(x_{s})\right\Vert^{2}_{\Gamma^{-1}}ds}{\int_{0}^{T}\left\Vert \nabla V(x_{s})\right\Vert^{2}_{I}ds}\right)^{-1}\hat{\theta}^{\epsilon}
\end{equation*}

converges in $\mathbb{P}_{\theta_{0}}$ probability to $\theta_{0}$, i.e., $\tilde{\theta}^{\epsilon}$ is a consistent estimator of $\theta_{0}$.
\end{corollary}
\begin{proof}
The separability assumption of $Q(y)$ gives us
\begin{equation*}
\frac{\partial \chi}{\partial y}(y)=\textrm{diag}\left(-1+\frac{\lambda}{\hat{Z}_{1}}e^{\frac{Q_{1}(y_{1})}{D}},\cdots, -1+\frac{\lambda}{\hat{Z}_{d}}e^{\frac{Q_{d}(y_{d})}{D}}\right).
\end{equation*}
Plugging that into (\ref{Eq:ErrorLangevinGeneral}) we immediately get the simplified representation of the error term.

The second claim follows from H\"{o}lder inequality. Indeed, it is easy to see that $\frac{\lambda^{2}}{Z\hat{Z}}\leq 1$. Therefore, we obtain $H_{\theta,\theta_{0}}(z_{\cdot})\leq 0$.

Next, we maximize the limiting log-likelihood function. By straightforward substitution to (\ref{Eq:MainTermRegime1general}) we see that
\begin{equation*}
J^{1}_{\theta,\theta_{0},T}(z_{\cdot})=\frac{1}{2}T\int_{\mathcal{Y}}\left\Vert \nabla Q\left(y\right)\right\Vert^{2}_{2D I}\mu(dy)
+\frac{1}{2D}\left(\theta\theta_{0}-\frac{1}{2}\theta^{2}\right)\int_{0}^{T}\left\Vert \nabla V\left(x_{s}\right)\right\Vert^{2}_{I}ds.
\end{equation*}
We collect things together and write
\begin{equation*}
\hat{Z}^{1}_{\theta,\theta_{0},T}(z_{\cdot})=\frac{1}{2}T\int_{\mathcal{Y}}\left\Vert  \nabla Q\left(y\right)\right\Vert^{2}_{2D I}\mu(dy)+
\frac{1}{2D}\left(-\frac{1}{2}\theta^{2}\int_{0}^{T}\left\Vert \nabla V(x_{s})\right\Vert^{2}_{I}ds+\theta\theta_{0}\lambda^{2}\int_{0}^{T}\left\Vert \nabla V(x_{s})\right\Vert^{2}_{\Gamma^{-1}}ds\right)
\end{equation*}
Then, it is easy to see that this quantity is maximized for
\begin{equation}
\hat{\theta}=\theta_{0}\lambda^{2}\frac{\int_{0}^{T}\left\Vert \nabla V(x_{s})\right\Vert^{2}_{\Gamma^{-1}}ds}{\int_{0}^{T}\left\Vert \nabla V(x_{s})\right\Vert^{2}_{I}ds}.\label{Eq:LangevinEquationEstimator}
\end{equation}

Then, using Theorem \ref{T:LikelihoodConvergence2} we obtain the statement of the theorem.
\end{proof}


\subsection{Central Limit Theorem for the pseudo MLE}\label{SS:LangevinEquationCLT}

In this section, we prove a central limit for the maximum likelihood estimator of the first order Langevin equation (\ref{Eq:LangevinEquation2}).

Based on the modified log likelihood function (i.e., on (\ref{Eq:Likelihood1})),  the maximum likelihood estimator can be written
\begin{eqnarray}\label{Eq:ModMle}
\hat{\theta}^{\epsilon} &=&  \left\{\left(\left(\frac{\delta}{\epsilon}\right)^{2}+1\right)\int_{0}^{T}\left\Vert\nabla V(x_{s})\right\Vert^{2}ds\right\}^{-1}
\cdot \Biggl\{\frac{\epsilon}{\delta}\int_{0}^{T}\left<\nabla V(x_{s}),\nabla Q(\frac{x_{s}}{\delta})\right>ds\nonumber \\
&& - \sqrt{2D}\sqrt{\epsilon}\left(\left(\frac{\delta}{\epsilon}\right)^{2}+1\right)\int_{0}^{T}\left<\nabla V(x_{s}),dW_{s}\right>  +\theta_{0}\left(\left(\frac{\delta}{\epsilon}\right)^{2}+1\right)\int_{0}^{T}\left\Vert\nabla V(x_{s})\right\Vert^{2}ds\Biggr\}.
\end{eqnarray}
Some algebra manipulation in (\ref{Eq:ModMle}) gives us
\begin{equation}
\frac{1}{\sqrt{\epsilon}}\left(\hat{\theta}^{\epsilon}-\theta_{0}-\frac{\frac{\epsilon}{\delta}\int_{0}^{T}\left<\nabla V(x_{s}),\nabla Q(\frac{x_{s}}{\delta})\right>ds}{\left(\left(\frac{\delta}{\epsilon}\right)^{2}+1\right)\int_{0}^{T}\left\Vert\nabla V(x_{s})\right\Vert^{2}ds} \right)=-
\sqrt{2D}\frac{\int_{0}^{T}\left<\nabla V(x_{s}),dW_{s}\right>
}{\int_{0}^{T}\left\Vert\nabla V(x_{s})\right\Vert^{2}ds}.\label{Eq:MLE_Langevin}
\end{equation}
We have the following theorem
\begin{theorem}
Assume Condition \ref{A:Assumption1}. Consider the first order Langevin equation (\ref{Eq:LangevinEquation2}) and assume Regime $1$.  Let $\hat{\theta}^{\epsilon}$ be the maximum likelihood estimator of $\theta_{0}$
based on the modified log likelihood function $\hat{Z}^{\epsilon}_{\theta,T}$.
Then, we have that in distribution under $\mathbb{P}_{\theta_{0}}$, the following central limit result holds
\begin{equation*}
\frac{1}{\sqrt{\epsilon}}\left(\hat{\theta}^{\epsilon}-\theta_{0}-\frac{\frac{\epsilon}{\delta}\int_{0}^{T}\left<\nabla V(x_{s}),\nabla Q(\frac{x_{s}}{\delta})\right>ds}{\left(\left(\frac{\delta}{\epsilon}\right)^{2}+1\right)\int_{0}^{T}\left\Vert\nabla V(x_{s})\right\Vert^{2}ds} \right)
\Rightarrow N\left(0,2D\left(\int_{0}^{T}\left\Vert\nabla V(\bar{X}^{1}_{s})\right\Vert^{2}ds\right)^{-1}\right).
\end{equation*}
Moreover, assuming $Q(y_{1},y_{2}%
,\cdots,y_{d})=Q_{1}(y_{1})+Q_{2}(y_{2})+\cdots+Q_{d}(y_{d})$, we also have that in $\mathbb{P}_{\theta_{0}}$ probability
\begin{equation*}
\lim_{\epsilon\downarrow 0}\left(\theta_{0}+\frac{\frac{\epsilon}{\delta}\int_{0}^{T}\left<\nabla V(x_{s}),\nabla Q(\frac{x_{s}}{\delta})\right>ds}{\left(\left(\frac{\delta}{\epsilon}\right)^{2}+1\right)\int_{0}^{T}\left\Vert\nabla V(x_{s})\right\Vert^{2}ds}\right)=
\theta_{0}\lambda^{2}\frac{\int_{0}^{T}\left\Vert \nabla V(\bar{X}^{1}_{s})\right\Vert^{2}_{\Gamma^{-1}}ds}{\int_{0}^{T}\left\Vert \nabla V(\bar{X}^{1}_{s})\right\Vert^{2}_{I}ds}.
\end{equation*}
which is (\ref{Eq:LangevinEquationEstimator}).
\end{theorem}
\begin{proof}
The first statement  follows directly from the representation of the maximum likelihood estimator in (\ref{Eq:MLE_Langevin}) and the central limit theorem for stochastic integrals, see for
example Lemma 1.8 in Chapter I of \cite{Kutoyants1994}.

The second statement is as follows. Consider the unique, bounded and periodic in $y$ smooth solution of the auxiliary problem
\begin{equation}
\mathcal{L}^{1}_{x}\Phi(x,y)=-\left<\nabla V(x),\nabla Q(y)\right>, \qquad \int_{\mathcal{Y}}\Phi(x,y)\mu(dy)=0.\label{Eq:PoissonEquationLangevin}
\end{equation}
By applying It\^{o} formula to $\Phi(x,y)$, (compare with (\ref{Eq:ItoFormula1}) and (\ref{Eq:ItoFormula2})) we get
\begin{eqnarray}
\frac{\epsilon}{\delta}\int_{0}^{T}\left<\nabla V\left(x_{s}\right), \nabla Q\left(\frac{x_{s}}{\delta}\right)\right>ds&=&-\theta_{0}\int_{0}^{T}\left<\nabla V(x_{s}),\nabla_{y}\Phi\left(x_{s},\frac{x_{s}}{\delta}\right)\right> ds +o_{p}(1).\nonumber
\end{eqnarray}
where the term $o_{p}(1)$ converges to zero in probability as $\epsilon,\delta\downarrow 0$. Therefore, by substituting we obtain that
\begin{equation*}
\frac{\frac{\epsilon}{\delta}\int_{0}^{T}\left<\nabla V(x_{s}),\nabla Q(\frac{x_{s}}{\delta})\right>ds}{\left(\left(\frac{\delta}{\epsilon}\right)^{2}+1\right)\int_{0}^{T}\left\Vert\nabla V(x_{s})\right\Vert^{2}ds}=
-\frac{\theta_{0}\int_{0}^{T}\left<\nabla V(x_{s}),\nabla_{y}\Phi\left(x_{s},\frac{x_{s}}{\delta}\right)\right> ds}{\left(\left(\frac{\delta}{\epsilon}\right)^{2}+1\right)\int_{0}^{T}\left\Vert\nabla V(x_{s})\right\Vert^{2}ds}+ o_{p}(1),
\end{equation*}
 Then, as in Proposition \ref{P:Langevin} and  Corollary \ref{C:corol} we can solve the auxiliary PDE (\ref{Eq:PoissonEquationLangevin})
in closed form and obtain the statement of the theorem.
\end{proof}

\subsection{A Simulation Study}\label{SS:LangevinEquationSimulation}

We apply our results in the case when $V^{\epsilon}(x,\frac{x}{\delta})=\epsilon \left(  \cos(\frac{x}{\delta}) +\sin(\frac{x}{\delta}) \right)  + \frac{1}{2}x^{2}$ and
$V(x)=\frac{1}{2}x^{2}$.

As we discussed in the previous sections, in this case we need to work with the modified likelihood, since $b\neq 0$. As we proved in Proposition 6.1 due to the separability of $Q$, we can obtain a consistent estimator when properly normalized.

We start by simulating the model. We use an Euler discretization scheme for the multiscale diffusion as follows
\[ X_{t_{k} +1}^{\epsilon} = X_{t_{k}}^{\epsilon} +  \left\{- \frac{\epsilon}{\delta} \left[ \cos\left(\frac{X_{t_{k}}^{\epsilon}}{\delta}\right) - \sin\left(\frac{X_{t_{k}}^{\epsilon}}{\delta}\right) \right] - \theta X_{t_{k}}^{\epsilon} \right\} (t_{k+1} - t_{k}) + \sqrt{\epsilon} \sqrt{2D} \left( W_{t_{k+1}} - W_{t_{k}} \right),   \]
where $k=1, \ldots, n$, $n$ is the number of simulated values. For the simulated data we choose $\epsilon=0.1$ and $\delta=0.01$.

From \cite{DupuisSpiliopoulosWang}, we have that the Euler scheme is bounded above by $\Delta \epsilon/\delta^{2}$, where we denote by $\Delta$ the discretization step. Therefore, if we want an error of order 0.001, we need to choose  the discretization step  $\Delta$ to be equal to $0.001\delta^{2}/\epsilon$. For the simulation procedure, we choose $\Delta=10^{-6}$ and $n=10^{6}$.

The maximum likelihood procedure consists of constructing the pseudo log-likelihood function (4.9). More specifically,
\begin{eqnarray*}
\hat{Z}_{\theta,T}^{\epsilon}(\mathcal{X}_{T})&=&\left(\frac{\delta}{\epsilon}\right)^{2}Z_{\theta,T}^{\epsilon}(\mathcal{X}_{T})+Z_{\theta,T}^{\epsilon}(\mathcal{X}_{T};0)\\
&=& \theta^{2} \left[ -\frac{1}{2\sigma^{2}} \int_{0}^{T} \left(\nabla V(x_{s})^{2} ds\right) \left( \left(\frac{\delta}{\epsilon}\right)^{2} + 1 \right) \right] \\
&& - \theta\; \frac{1}{\sigma^{2}} \left[ \int_{0}^{T} \nabla V(x_{s})dx_{s} \left( 1 + \left(\frac{\delta}{\epsilon}\right)^{2} \right) + \frac{\delta}{\epsilon} \int_{0}^{T} \nabla Q\left(\frac{x_{s}}{\delta}\right) \nabla V(x_{s}) ds \right] + const,
\end{eqnarray*}
where $\sigma=\sqrt{2D}$ and $const$ is a quantity that is independent of the parameter $\theta$. The maximizer of this quantity computes as
\begin{equation*}
\theta_{max} = \frac{\left[ \int_{0}^{T} \nabla V(x_{s})dx_{s} \left( 1 + \left(\frac{\delta}{\epsilon}\right)^{2} \right) + \frac{\delta}{\epsilon} \int_{0}^{T} \nabla Q\left(\frac{x_{s}}{\delta}\right) \nabla V(x_{s}) ds \right]}{- \int_{0}^{T} \left(\nabla V(x_{s})^{2} ds\right) \left( \left(\frac{\delta}{\epsilon}\right)^{2} + 1 \right) }.
\end{equation*}

Although our model is continuous as well as our MLE, in practice we obtain data in discrete time. Therefore, we need to discretize our estimator in order to implement it.
We directly discretize the stochastic integrals and we obtain
\begin{equation*}
\hat{\theta}_{max} = \frac{\left[ \sum_{i=1}^{N-1} \nabla V(x_{s_i})(x_{s_{i+1}} - x_{s_i}) \left( 1 + \left(\frac{\delta}{\epsilon}\right)^{2} \right) + \frac{\delta}{\epsilon} \sum_{i=1}^{N} \nabla Q\left(\frac{x_{s_i}}{\delta}\right) \nabla V(x_{s_i}) (s_{i+1}-s_i) \right]}{- \sum_{i=1}^{N-1} \left(\nabla V(x_{s_i})^{2} (s_{i+1}-s_i)\right) \left( \left(\frac{\delta}{\epsilon}\right)^{2} + 1 \right) }.
\end{equation*}
The consistent estimator $\hat{\theta}$ will be the normalized $\hat{\theta}_{max}$. The normalizing term equals $(\lambda \Gamma)^{2}$, with $\lambda$ and $\Gamma$ as defined in Corollary (\ref{C:corol}).
\begin{remark}
It is important to mention here that we do not simplify the stochastic integral using It\^o's lemma. The reason is that in order to compute the estimator for $\theta$ we need to use just the observations we have available. If we use It\^o, then the integral with respect to Brownian motion that appears contains a process (the Brownian motion) that is not observed.
\end{remark}

Using simulated data, we construct the MLE for different values of the true parameter $\theta$. The results are summarized in Table \ref{T:Table}, along with the corresponding 68\% and 95\% confidence intervals. These are both empirical intervals meaning that we repeat the procedure (simulation -- estimation) several times (M=100). Then, we obtain the Monte Carlo estimator for $\theta$ as the average of all estimators, as well as the Monte Carlo standard deviation that we use in the construction of the intervals.

\begin{table}[!h]
\label{T:Table}
\begin{center}
\begin{tabular}{|c||c|c|c|}
\hline
\textit{True Value} & \textit{Estimator} & \textit{68\% Confidence Interval} & \textit{95\% Confidence Interval}\\ \hline \hline
1   &  1.042 & ( -0.0329, 2.118) & (-1.065, 3.150) \\ \hline
2   &  1.970 & ( 0.1827, 3.758) & (-1.533, 5.473) \\ \hline
0.1 &  0.103 & ( 0.0125, 0.1928) & (-0.0739, 0.2795) \\ \hline
\end{tabular}
\end{center}
\vspace{0.1cm}
\caption{Estimated values of $\theta$ and the corresponding empirical 68\% and 95\% confidence intervals for various true parameters $\theta$.}
\end{table}

For $\theta=1$, we  plot (Figure \ref{F:Figure2}) the  histogram of the empirical distribution that we obtain from the Monte Carlo procedure. For comparison, on the same graph we also plot the corresponding density curve of the theoretical asymptotic (Normal) distribution with the appropriate variance as the one we computed in Theorem 6.3.

\begin{figure}[!h]
\label{F:Figure2}
\par
\begin{center}
\includegraphics[scale=0.6, width=10 cm, height=8 cm]{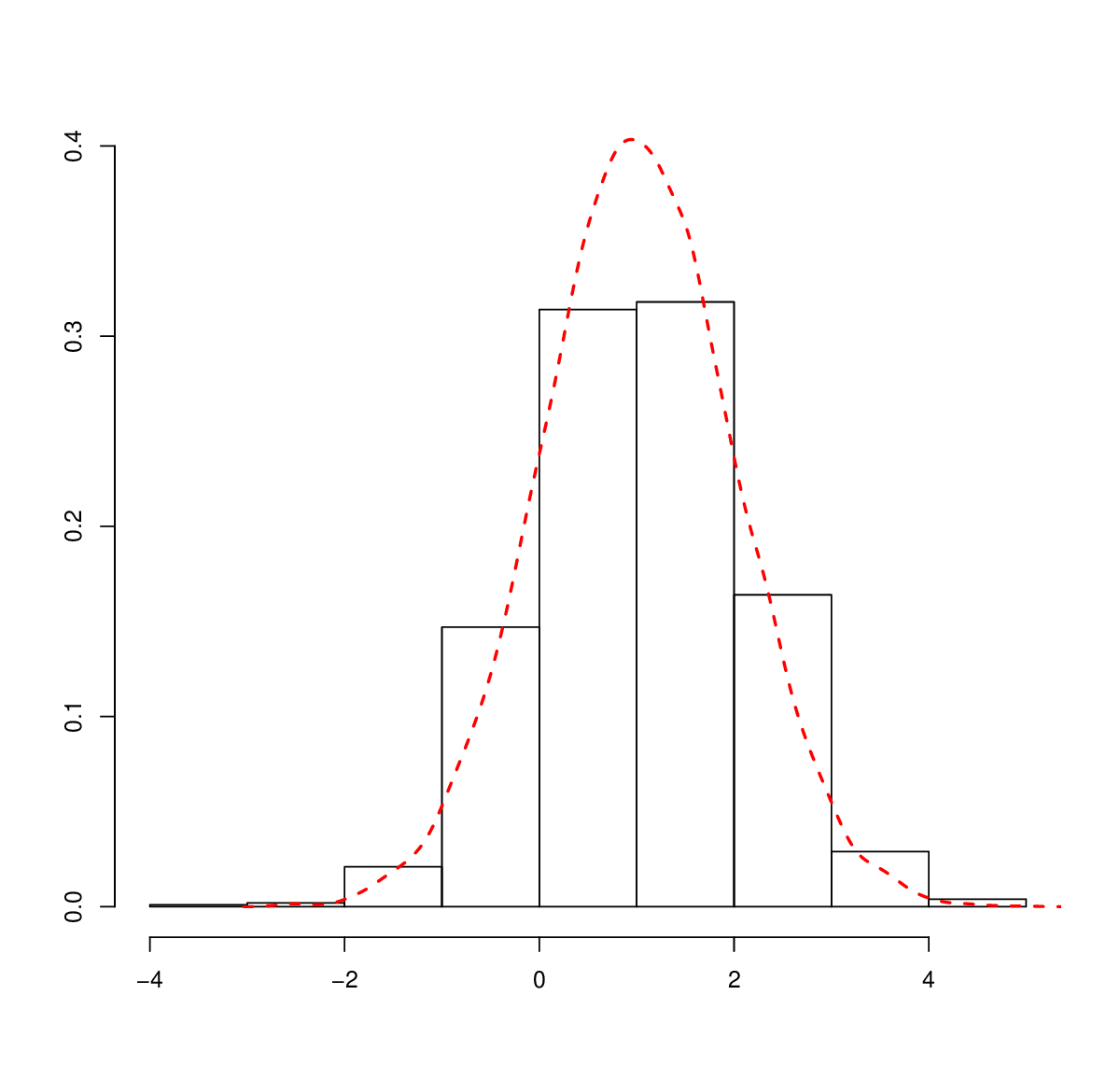}
\end{center}
\caption{Histogram of the estimator $\hat{\theta}_{max}$ for the simulated dataset and the corresponding density (theoretical) curve from Theorem 6.3.}
\end{figure}

\section{Conclusions}\label{S:Conclusions}

In this paper we studied the parameter estimation problem for diffusion processes with multiple scales and vanishing noise. Under certain conditions,
we derived consistent estimators and proved the related central limit theorems. The theoretical results are supported by a simulation study of the first order Langevin equation
 in a rough potential. Such results are useful when one is interested in parameter estimation of dynamical systems with more than one scales (e.g., in rough potentials) perturbed by small noise.

\section*{Appendix}

\begin{proof}[Proof of Lemma 5.3]
Let us denote $\theta_{\epsilon,u}=\theta_{\epsilon}+\phi(\epsilon,\theta_{\epsilon}) u_{\epsilon}$, where $\phi(\epsilon,\theta)=\sqrt{\epsilon}I^{-1/2}(\theta)$. We assume that  $\theta_{\epsilon}$ belongs in a compact subset of $\Theta$, denoted by $\tilde{\Theta}$, and let $u_{\epsilon}\rightarrow u$ as $\epsilon\downarrow 0$. We start by rewriting the normed likelihood ratio as follows
\begin{eqnarray}
M_{\epsilon}(\theta_{\epsilon},u)&=&\frac{1}{\sqrt{\epsilon}}\int_{0}^{T}
\left< c_{\theta_{\epsilon,u}}-c_{\theta_{\epsilon}}, \sigma dW_{s}\right>_{\alpha}\left(x_{s},\frac{x_{s}}{\delta}\right)-
\frac{1}{2\epsilon}\int_{0}^{T}\left\Vert c_{\theta_{\epsilon,u}}-c_{\theta_{\epsilon}}\right\Vert^{2}_{\alpha}\left(x_{s},\frac{x_{s}}{\delta}\right) ds\nonumber\\
&=&
\frac{1}{\sqrt{\epsilon}}\int_{0}^{T}
\left< c_{\theta_{\epsilon,u}}-c_{\theta_{\epsilon}}-\left(\sqrt{\epsilon}I^{-1/2}(\theta_{\epsilon})u_{\epsilon},\nabla_{\theta}c_{\theta_{\epsilon,u}}\right), \sigma dW_{s}\right>_{\alpha}\left(x_{s},\frac{x_{s}}{\delta}\right)\nonumber\\
& &+\left(I^{-1/2}(\theta_{\epsilon,u})u_{\epsilon},\int_{0}^{T}
\left<\nabla_{\theta}c_{\theta_{\epsilon,u}}, \sigma dW_{s}\right>_{\alpha}\left(x_{s},\frac{x_{s}}{\delta}\right)\right)\nonumber\\
& &-\frac{1}{2}\int_{0}^{T}\left[\frac{1}{\epsilon}\left\Vert c_{\theta_{\epsilon,u}}-c_{\theta_{\epsilon}}\right\Vert^{2}_{\alpha}\left(x_{s},\frac{x_{s}}{\delta}\right)-
\left(I^{-1/2}(\theta_{\epsilon})u_{\epsilon},q^{1/2}(\bar{X}_{s},\theta_{\epsilon})\right)^{2}\right]ds
-\frac{1}{2}\left(u_{\epsilon},u_{\epsilon}\right)\nonumber\\
&=&J^{\epsilon}_{1}(\theta_{\epsilon})+J^{\epsilon}_{2}(\theta_{\epsilon})+J^{\epsilon}_{3}(\theta_{\epsilon})+J^{\epsilon}_{4}.\nonumber
\end{eqnarray}
The last line of the previous computation is easily seen to hold by the following chain of identities
\begin{equation*}
\int_{0}^{T}\int_{\mathcal{Y}}\left(v,S(\theta_{\epsilon},\bar{X}_{s},y)\right)^{2}\mu_{\theta_{\epsilon}}(dy;\bar{X}_{s})ds=\int_{0}^{T}\left(v,q^{1/2}(\bar{X}_{s},\theta_{\epsilon})\right)^{2}ds=(I(\theta_{\epsilon})v,v).
\end{equation*}
which are applied for $v=I^{-1/2}(\theta_{\epsilon})u_{\epsilon}$.

The goal is to prove that $M_{\epsilon}(\theta_{\epsilon},u)=\left(u,\Phi\right)-\frac{1}{2}\left\Vert u\right\Vert^{2}+R(\epsilon,\theta_{\epsilon})$, where $\Phi$ is distributed as normal $N(0,I)$
and $R(\epsilon,\theta)\rightarrow 0$ as $\epsilon\downarrow 0$ in $\mathbb{P}_{\theta}$ probability uniformly in $\theta\in\Theta$. This, will establish that the family
$\{\mathbb{P}^{\epsilon}_{\theta}: \theta\in \Theta\}$ is uniformly asymptotically normal with normalizing matrix $\phi(\epsilon,\theta)=\sqrt{\epsilon}I^{-1/2}(\theta)$,
which then proves the lemma.

It is clear that
\begin{equation*}
J^{\epsilon}_{4}=-\frac{1}{2}\left(u_{\epsilon},u_{\epsilon}\right)\rightarrow-\frac{1}{2}\left\Vert u\right\Vert^{2}, \textrm{ as }\epsilon\downarrow 0.
\end{equation*}

Moreover, due to averaging and the law of large numbers result Theorem \ref{T:LLN}, the definition of the Fisher information matrix $I(\theta)$ implies that
\begin{align}
J^{\epsilon}_{2}(\theta_{\epsilon})&=\left(I^{-1/2}(\theta_{\epsilon,u})u_{\epsilon},\int_{0}^{T}
\left<\nabla_{\theta}c_{\theta_{\epsilon,u}}, \sigma dW_{s}\right>_{\alpha}\left(x_{s},\frac{x_{s}}{\delta}\right)\right)\nonumber\\
&= \left(I^{-1/2}(\theta_{\epsilon,u})u_{\epsilon},\int_{0}^{T}
\left<S\left(\theta_{\epsilon,u},x_{s},\frac{x_{s}}{\delta}\right),  dW_{s}\right>\right)\nonumber
\end{align}
converges in distribution with respect to $\mathbb{P}_{\theta}$, uniformly in $\theta\in\tilde{\Theta}$, to $\left(u,\Phi\right)$ where $\Phi$ is distributed as $N(0,I)$, as $\epsilon\downarrow 0$.

Thus it remains to consider the term $R(\epsilon,\theta)=J^{\epsilon}_{1}(\theta)+J^{\epsilon}_{3}(\theta)$.
We shall show that both terms converge to zero in $\mathbb{P}_{\theta}$ probability as $\epsilon\downarrow 0$, uniformly in $\theta\in\tilde{\Theta}$ .

We start by observing that
\begin{equation*}
c_{\theta+\ell}-c_{\theta}=\int_{0}^{1}\left(\ell,\nabla_{\theta}c_{\theta+\ell h}\right)dh.
\end{equation*}

Then we can write
\begin{align}
&\mathbb{E} \int_{0}^{T}\left\Vert \sigma^{-1}\left(\frac{1}{\sqrt{\epsilon}}( c_{\theta_{\epsilon,u}}-c_{\theta_{\epsilon}})-\left(I^{-1/2}(\theta_{\epsilon,u})u_{\epsilon},\nabla_{\theta}c_{\theta_{\epsilon,u}}\right)\right)\left(X^{\epsilon}_{s},\frac{X^{\epsilon}_{s}}{\delta}\right)\right\Vert^{2}ds
\nonumber\\
=\;& \mathbb{E}\int_{0}^{T}\left\Vert \left[\sigma^{-1}\int_{0}^{1}\left(I^{-1/2}(\theta_{\epsilon,u})u_{\epsilon},\nabla_{\theta}c_{\theta_{\epsilon,u}+h\sqrt{\epsilon}I^{-1/2}(\theta_{\epsilon,u})u_{\epsilon}}-\nabla_{\theta}c_{\theta_{\epsilon,u}}\right)dh\right]\left(X^{\epsilon}_{s},\frac{X^{\epsilon}_{s}}{\delta}\right)\right\Vert^{2}ds\nonumber\\
\leq\;& |I^{-1/2}(\theta_{\epsilon,u})u_{\epsilon}|^{2}
\sup_{\theta\in\tilde{\Theta}}\sup_{|v|\leq C\sqrt{\epsilon}}
\mathbb{E}\left|\int_{0}^{T}\left\Vert\nabla_{\theta}c_{\theta+v}-\nabla_{\theta}c_{\theta}\right\Vert^{2}_{\alpha}\left(X^{\epsilon}_{s},\frac{X^{\epsilon}_{s}}{\delta}\right)ds\right|\nonumber\\
\leq\;& C
\sup_{\theta\in\tilde{\Theta}}\sup_{|v|\leq C\sqrt{\epsilon}}
\mathbb{E}\left|\int_{0}^{T}\left\Vert\nabla_{\theta}c_{\theta+v}-\nabla_{\theta}c_{\theta}\right\Vert^{2}_{\alpha}\left(X^{\epsilon}_{s},\frac{X^{\epsilon}_{s}}{\delta}\right)ds\right|\rightarrow 0, \textrm{ as }\epsilon\downarrow 0.\label{Eq:J3_term1}
\end{align}
The last convergence is true due to the uniform continuity of $\nabla_{\theta}c_{\theta}$ in $\theta\in\tilde{\Theta}$ and tightness of $\left\{X^{\epsilon},\epsilon>0\right\}$.
Using It\^{o} isometry, the last display implies that
\begin{equation}
\sup_{\theta\in\tilde{\Theta}}\mathbb{E}\left|J^{\epsilon}_{1}(\theta)\right|^{2}\rightarrow 0,  \textrm{ as }\epsilon\downarrow 0. \label{Eq:J1}
\end{equation}

Lastly, it remains to consider the term $J^{3}(\theta)$. Notice that standard averaging principle, the convergence of $X^{\epsilon}$ to $\bar{X}$ as $\epsilon\downarrow 0$ by Theorem \ref{T:LLN}, and the continuous dependence of the involved functions on $\theta$, imply that,
\begin{equation}
\mathbb{E}\left|\int_{0}^{T}\left[
\left\Vert\left(I^{-1/2}(\theta_{\epsilon,u})u_{\epsilon},\nabla_{\theta}c_{\theta_{\epsilon,u}}\right)\right\Vert^{2}_{\alpha}\left(X^{\epsilon}_{s},\frac{X^{\epsilon}_{s}}{\delta}\right)
-
\left(I^{-1/2}(\theta_{\epsilon,u})u_{\epsilon},q^{1/2}(\bar{X}_{s},\theta_{\epsilon,u})\right)^{2}\right]ds\right|\rightarrow 0, \textrm{ as }\epsilon\downarrow 0. \label{Eq:J3_term2}
\end{equation}
By (\ref{Eq:J3_term1})-(\ref{Eq:J3_term2}) and the assumptions on the dependence on $\theta$  we obtain that
\begin{equation*}
\sup_{\theta\in\tilde{\Theta}}\mathbb{E}\left|J^{\epsilon}_{3}(\theta)\right|^{2}\rightarrow 0,  \textrm{ as }\epsilon\downarrow 0.\label{Eq:J3}
\end{equation*}

Therefore, we have obtained that
\begin{equation*}
 \sup_{\theta\in\tilde{\Theta}}\mathbb{E}\left|R(\epsilon,\theta)\right|^{2}\rightarrow 0, \textrm{ as }\epsilon\downarrow 0.
\end{equation*}

This establishes that the family $\{\mathbb{P}^{\epsilon}_{\theta}: \theta\in \Theta\}$ is uniformly asymptotically normal with normalizing matrix $\phi(\epsilon,\theta)=\sqrt{\epsilon}I^{-1/2}(\theta)$, which concludes the proof of the lemma.
\end{proof}

\begin{proof}[Proof of Lemma 5.4]
The proof follows along the lines of Lemma 2.3 in \cite{Kutoyants1994}. We review it here for completeness and mention the required modifications in order to account for the extra
component of averaging. Let $\theta_{i}=\theta+\phi(\epsilon,\theta) u_{i}$ and define the interpolating point
\[
\theta(\ell)=\theta_{1}+(\theta_{2}-\theta_{1})\ell, \quad, \ell\in[0,1]
\]
By an absolutely continuous change of measure we have
\begin{equation*}
\mathbb{E}_{\theta}\left|e^{\frac{1}{2m}M_{\epsilon}(\theta,u_{2})}-e^{\frac{1}{2m}M_{\epsilon}(\theta,u_{1})} \right|^{2m}=\mathbb{E}_{\theta_{1}}\left|L^{\frac{1}{2m}}(\theta_{2},\theta_{1};x)-1 \right|^{2m}
\end{equation*}
where, we have defined $ L(\theta_{2},\theta_{1};x)=\frac{d\mathbb{P}_{\theta_{2}}}{d\mathbb{P}_{\theta_{1}}}(x)$.  Then, we write
\begin{align*}
&\log L(\theta_{2},\theta_{1};x)=\frac{1}{\sqrt{\epsilon}}\int_{0}^{T}
\left< c_{\theta_{2}}-c_{\theta_{1}}, \sigma dW_{s}\right>_{\alpha}\left(X^{\epsilon}_{s},\frac{X^{\epsilon}_{s}}{\delta}\right)-
\frac{1}{2\epsilon}\int_{0}^{T}\left\Vert c_{\theta_{2}}-c_{\theta_{1}}\right\Vert^{2}_{\alpha}\left(X^{\epsilon}_{s},\frac{X^{\epsilon}_{s}}{\delta}\right) ds\nonumber\\
&=\frac{1}{\sqrt{\epsilon}}\int_{0}^{1}\int_{0}^{T}
\left< \nabla_{\theta}c_{\theta(\ell)}(\theta_{2}-\theta_{1}), \sigma dW_{s}\right>_{\alpha}\left(X^{\epsilon}_{s},\frac{X^{\epsilon}_{s}}{\delta}\right)d\ell\nonumber\\
&\quad-
\frac{1}{2\epsilon}\int_{0}^{1}\int_{0}^{T}\left< c_{\theta_{2}}-c_{\theta_{1}}, \nabla_{\theta} c_{\theta(\ell)}(\theta_{2}-\theta_{1})\right>_{\alpha}\left(X^{\epsilon}_{s},\frac{X^{\epsilon}_{s}}{\delta}\right) dsd\ell\nonumber\\
&=2m\int_{0}^{1}\left<f(\theta(\ell)),\theta_{2}-\theta_{1}\right>d\ell,\nonumber
\end{align*}
where
\[
f(\theta(\ell))=\frac{1}{\sqrt{\epsilon}}\int_{0}^{T}
\left< \nabla_{\theta}c_{\theta(\ell)}, \sigma dW_{s}\right>_{\alpha}\left(X^{\epsilon}_{s},\frac{X^{\epsilon}_{s}}{\delta}\right)-
\frac{1}{2\epsilon}\int_{0}^{T}\left< c_{\theta_{2}}-c_{\theta_{1}}, \nabla_{\theta} c_{\theta(\ell)}\right>_{\alpha}\left(X^{\epsilon}_{s},\frac{X^{\epsilon}_{s}}{\delta}\right) ds
\]
Thus, we obtain
\begin{align*}
\mathbb{E}_{\theta_{1}}\left|L^{\frac{1}{2m}}(\theta_{2},\theta_{1};x)-1 \right|^{2m}&\leq \mathbb{E}_{\theta_{1}}\left|\int_{0}^{1}\left<f(\theta(\ell)),\theta_{2}-\theta_{1}\right> e^{\int_{0}^{\ell} \left<f(\theta(\kappa)),\theta_{2}-\theta_{1}\right>d\kappa}d\ell \right|^{2m}\nonumber\\
&\leq \int_{0}^{1}\mathbb{E}_{\theta_{1}}\left[L(\theta_{2},\theta_{1};x)\left<f(\theta(\ell)),\theta_{2}-\theta_{1}\right>^{2m}  \right]d\ell\nonumber\\
&= \epsilon^{-m} (2m)^{-2m} \int_{0}^{1}\mathbb{E}_{\theta_{2}}\left|\int_{0}^{T}\left<\nabla_{\theta}c_{\theta(\ell)}(\theta_{2}-\theta_{1}),\sigma dW_{s}\right>_{\alpha}\left(X^{\epsilon}_{s},\frac{X^{\epsilon}_{s}}{\delta}\right) \right|^{2m}d\ell\nonumber\\
&\leq \epsilon^{-m} C_{m,T} \int_{0}^{1}\mathbb{E}_{\theta_{2}}\left[\int_{0}^{T}\left<\nabla_{\theta}c_{\theta(\ell)}\left(X^{\epsilon}_{s},\frac{X^{\epsilon}_{s}}{\delta}\right),(\theta_{2}-\theta_{1})\right>_{\alpha}^{2m}ds\right]d\ell\nonumber\\
&\leq \epsilon^{-m} C_{m,T} |\theta_{2}-\theta_{1}|^{2m}\sup_{\theta_{2},\theta\in\tilde{\Theta}}\mathbb{E}_{\theta_{2}}\left[\int_{0}^{T}\left\Vert\nabla_{\theta}c_{\theta}\right\Vert_{\alpha}^{2m}\left(X^{\epsilon}_{s},\frac{X^{\epsilon}_{s}}{\delta}\right)ds\right]\nonumber\\
&\leq \epsilon^{-m} \phi^{2m}(\epsilon,\theta) C_{m,T} |u_{2}-u_{1}|^{2m}\sup_{\theta_{2},\theta\in\tilde{\Theta}}\mathbb{E}_{\theta_{2}}\left[\int_{0}^{T}\left\Vert\nabla_{\theta}c_{\theta}\right\Vert_{\alpha}^{2m}\left(X^{\epsilon}_{s},\frac{X^{\epsilon}_{s}}{\delta}\right)ds\right]\nonumber\\
&\leq  I^{-m}(\theta) C_{m,T} |u_{2}-u_{1}|^{2m}\sup_{\theta_{2},\theta\in\tilde{\Theta}}\mathbb{E}_{\theta_{2}}\left[\int_{0}^{T}\left\Vert\nabla_{\theta}c_{\theta}\right\Vert_{\alpha}^{2m}\left(X^{\epsilon}_{s},\frac{X^{\epsilon}_{s}}{\delta}\right)ds\right]\nonumber
\end{align*}
and the result follows by the assumed uniform boundedness of $\sigma^{-1}\nabla_{\theta}c_{\theta}(x,y)$.
\end{proof}

\begin{proof}[Proof of Lemma 5.5]
In the absence of multiple scales, this is Lemma 2.4 in \cite{Kutoyants1994}. Here we provide the proof of the result with the additional component of multiple scales, which makes the analysis more involved. For the sake of concreteness we only present the proof for the case of Regime $1$. The required changes for the other regimes are minimal and are mentioned below at the appropriate place.

Recall that $\phi(\epsilon,\theta)=\sqrt{\epsilon}I^{-1/2}(\theta)$ and set
\[
\Delta c_{\theta}(x,y)=\frac{1}{\sqrt{\epsilon}}\left[c_{\theta+\phi(\epsilon,\theta)u}(x,y)-c_{\theta}(x,y)\right]
\]
We can then write
\begin{align*}
\mathbb{E}_{\theta}e^{pM_{\epsilon}(\theta,u)}&=\mathbb{E}_{\theta}\left[e^{p\int_{0}^{T}\left<\Delta c_{\theta},\sigma dW_{s}\right>_{\alpha}\left(X^{\epsilon}_{s},\frac{X^{\epsilon}_{s}}{\delta}\right)-\frac{p}{2}\int_{0}^{T}\left\Vert \Delta c_{\theta}\right\Vert_{\alpha}^{2}\left(X^{\epsilon}_{s},\frac{X^{\epsilon}_{s}}{\delta}\right)ds}\right]\nonumber\\
&\leq \left(\mathbb{E}_{\theta}e^{-p_1\cdot \frac{p-q}{2}\int_{0}^{T}\left\Vert \Delta c_{\theta}\right\Vert_{\alpha}^{2}\left(X^{\epsilon}_{s},\frac{X^{\epsilon}_{s}}{\delta}\right)ds}\right)^{1/p_{1}}\times \nonumber\\
&\quad\times \left(\mathbb{E}_{\theta}\left[e^{pp_{2}\int_{0}^{T}\left<\Delta c_{\theta},\sigma dW_{s}\right>_{\alpha}\left(X^{\epsilon}_{s},\frac{X^{\epsilon}_{s}}{\delta}\right)-\frac{qp_{2}}{2}\int_{0}^{T}\left\Vert \Delta c_{\theta}\right\Vert_{\alpha}^{2}\left(X^{\epsilon}_{s},\frac{X^{\epsilon}_{s}}{\delta}\right)ds}\right]\right)^{1/p_{2}}
\end{align*}
Choosing now $p_{2}=q/p^{2}>1$, we have that
\[
\mathbb{E}_{\theta}\left[e^{pp_{2}\int_{0}^{T}\left<\Delta c_{\theta},\sigma dW_{s}\right>_{\alpha}\left(X^{\epsilon}_{s},\frac{X^{\epsilon}_{s}}{\delta}\right)-\frac{qp_{2}}{2}\int_{0}^{T}\left\Vert \Delta c_{\theta}\right\Vert_{\alpha}^{2}\left(X^{\epsilon}_{s},\frac{X^{\epsilon}_{s}}{\delta}\right)ds}\right]\leq 1
\]
Setting $\gamma=q\frac{p-q}{2(q-p^{2})}>0$, this implies
\begin{align}
\mathbb{E}_{\theta}e^{pM_{\epsilon}(\theta,u)}
&\leq \left(\mathbb{E}_{\theta}e^{-\gamma\int_{0}^{T}\left\Vert \Delta c_{\theta}\right\Vert_{\alpha}^{2}\left(X^{\epsilon}_{s},\frac{X^{\epsilon}_{s}}{\delta}\right)ds}\right)^{(q-p^{2})/p}\label{Eq:CrucialTermToBound}
\end{align}
So, the next step is to appropriately bound from above the term $\mathbb{E}_{\theta}e^{-\gamma\int_{0}^{T}\left\Vert \Delta c_{\theta}\right\Vert_{\alpha}^{2}\left(X^{\epsilon}_{s},\frac{X^{\epsilon}_{s}}{\delta}\right)ds}$.

At this point, we recall the definition of $q(x,\theta)$ from (\ref{Eq:Definition_q})
 and we write
\begin{eqnarray}
q_{\epsilon,v}(x,\theta)&=&\int_{\mathcal{Y}}\left(\int_{0}^{1}S(\theta+\sqrt{\epsilon}v h,x,y)dh\right)\left(\int_{0}^{1}S(\theta+\sqrt{\epsilon}v h,x,y)dh\right)^{T}\mu_{\theta}(dy;x).\nonumber
\end{eqnarray}
Define the operator
\[
\mathcal{L}_{x} =\frac{1}{2}\sigma(x,y)\sigma^{T}(x,y):\nabla_{y}\nabla_{y}
\]
and for $v=I^{-1/2}(\theta)u$, let $\Phi=\Phi_{\theta,\epsilon,v}(x,y)$ satisfy the auxiliary PDE
\begin{equation}
\mathcal{L}_{x}\Phi(x,y)=\left\Vert \Delta c_{\theta}(x,y)\right\Vert^{2}_{\alpha}-\left(v,q_{\epsilon,v}^{1/2}(x,\theta)\right)^{2}, \qquad \int_{\mathcal{Y}}\Phi(x,y)\mu_{\theta}(dy;x)=0.\label{Eq:PDE_CLT}
\end{equation}
Comparing with the case without the multiple scales, the additional difficulty here is the presence of the fast oscillating component, $X^{\epsilon}/\delta$. The consideration of the solution to this auxiliary
PDE, allows us to reduce the bound for the quantity at hand to a bound for a quantity that depends only on the slow component, $X^{\epsilon}$.

Notice that $\mathcal{L}_{x}$ is the operator for Regime $1$ defined in Definition \ref{Def:ThreePossibleOperators} with $b=0$. For Regimes 2 and 3, one would need to consider the solution to the PDE
governed by the corresponding operators from Definition  \ref{Def:ThreePossibleOperators}. Since,
\begin{align*}
 \int_{\mathcal{Y}}\left\Vert \Delta c_{\theta}(x,y)\right\Vert^{2}_{\alpha}\mu_{\theta}(dy;x)&=
\int_{\mathcal{Y}}\left\Vert \int_{0}^{1}\left(v, \nabla_{\theta}c_{\theta+\sqrt{\epsilon}vh}(x,y)\right)dh\right\Vert^{2}_{\alpha}\mu_{\theta}(dy;x)\nonumber\\
&=\int_{\mathcal{Y}} \left|\left(v, \int_{0}^{1} \sigma^{-1}\nabla_{\theta}c_{\theta+\sqrt{\epsilon}vh}(x,y)dh\right)\right|^{2}
\mu_{\theta}(dy;x)\nonumber\\
&=\left(v, q^{1/2}_{\epsilon,v}(x,\theta)\right)^{2}\nonumber
\end{align*}
 Fredholm alternative, Theorem 3.3.4 of \cite{BLP} guarantees that there exists a unique, smooth, periodic in $y$ and bounded solution to the aforementioned auxiliary PDE for $\Phi$.
The boundedness of $\Theta$ and the imposed conditions on $\nabla_{\theta}c_{\theta}$ also guarantee that $\Phi$ is bounded uniformly in $\theta,(\theta+\sqrt{\epsilon}v)\in\Theta$. Let us apply It\^{o}
formula to $\Phi(x,x/\delta)$ with $x=X^{\epsilon}_{s}$. It\^{o} formula gives an expression similar to (\ref{Eq:ItoFormula1}) and after some term rearrangement, we get for $\theta\in\tilde{\Theta}$ that
\begin{align}
 &\int_{0}^{T}\left[\left\Vert \Delta c_{\theta}\left(X^{\epsilon}_{s},\frac{X^{\epsilon}_{s}}{\delta}\right)\right\Vert^{2}_{\alpha}-\left(v,q_{\epsilon,v}^{1/2}\left(X^{\epsilon}_{s},\theta\right)\right)^{2}\right]ds=
\int_{0}^{T}\mathcal{L}_{X^{\epsilon}_{s}}\Phi\left(X^{\epsilon}_{s},\frac{X^{\epsilon}_{s}}{\delta}\right)ds\nonumber\\
&\quad= (\delta^{2}/\epsilon)\left(\Phi\left(X^{\epsilon}_{t},\frac{X^{\epsilon}_{t}}{\delta}\right)-\Phi\left(X^{\epsilon}_{0},\frac{X^{\epsilon}_{0}}{\delta}\right)\right)\nonumber\\
&\quad -\int_{0}^{T}\left[\frac{\delta}{\epsilon}\left<c_{\theta},\nabla_{y}\Phi\right>+ \frac{\delta^{2}}{\epsilon}\left<c_{\theta},\nabla_{x}\Phi\right>
+\frac{\delta^{2}}{2}\sigma\sigma^{T}:\nabla_{x}\nabla_{x}\Phi+\delta\sigma\sigma^{T}:\nabla_{x}\nabla_{y}\Phi\right]\left(X^{\epsilon}_{s},\frac{X^{\epsilon}_{s}}{\delta}\right)ds\nonumber\\
&\quad - \frac{\delta}{\sqrt{\epsilon}}\int_{0}^{T}\left<\nabla_{y}\Phi,\sigma dW_{s}\right>\left(X^{\epsilon}_{s},\frac{X^{\epsilon}_{s}}{\delta}\right)-
\frac{\delta^{2}}{\sqrt{\epsilon}}\int_{0}^{T}\left<\nabla_{x}\Phi,\sigma dW_{s}\right>\left(X^{\epsilon}_{s},\frac{X^{\epsilon}_{s}}{\delta}\right)\label{Eq:PDE_CLT2}
\end{align}

Due to the boundedness of the involved functions, the last display gives us the existence of a constant $C$ that may depend on $\tilde{\Theta}$ (but not on $(\epsilon,\delta)\in(0,1)^{2}$), such that
\begin{align}
\left|\int_{0}^{T}\left[\left\Vert \Delta c_{\theta}\left(X^{\epsilon}_{s},\frac{X^{\epsilon}_{s}}{\delta}\right)\right\Vert^{2}_{\alpha}-\left(v,q_{\epsilon,v}^{1/2}\left(X^{\epsilon}_{s},\theta\right)\right)^{2}\right]ds\right|\leq
C \left(1+\sup_{t\in[0,T]}|W_{t}|\right) \label{Eq:CrucialTermToBound1}
\end{align}

These computations, allow us to continue the right hand side of (\ref{Eq:CrucialTermToBound}) as follows
\begin{align}
\mathbb{E}_{\theta}e^{-\gamma\int_{0}^{T}\left\Vert \Delta c_{\theta}\right\Vert_{\alpha}^{2}\left(X^{\epsilon}_{s},\frac{X^{\epsilon}_{s}}{\delta}\right)ds}&=
\mathbb{E}_{\theta}\left\{e^{-\gamma\left[\int_{0}^{T}\left[\left\Vert \Delta c_{\theta}\left(X^{\epsilon}_{s},\frac{X^{\epsilon}_{s}}{\delta}\right)\right\Vert^{2}_{\alpha}-\left(v,q_{\epsilon,v}^{1/2}\left(X^{\epsilon}_{s},\theta\right)\right)^{2}\right]ds\right]}\times\right.\nonumber\\
&\quad\qquad\left.\times e^{-\gamma\int_{0}^{T}\left(v,q_{\epsilon,v}^{1/2}\left(X^{\epsilon}_{s},\theta\right)\right)^{2}ds}\right\}\nonumber\\
&\leq \left(\mathbb{E}_{\theta}e^{-\gamma p_{3}\int_{0}^{T}\left(v,q_{\epsilon,v}^{1/2}\left(X^{\epsilon}_{s},\theta\right)\right)^{2}ds}\right)^{1/p_{3}}\times\nonumber\\
&\quad \times  \left(\mathbb{E}_{\theta}e^{-\gamma q_{3}\left[\int_{0}^{T}\left[\left\Vert \Delta c_{\theta}\left(X^{\epsilon}_{s},\frac{X^{\epsilon}_{s}}{\delta}\right)\right\Vert^{2}_{\alpha}-\left(v,q_{\epsilon,v}^{1/2}\left(X^{\epsilon}_{s},\theta\right)\right)^{2}\right]ds\right]}\right)^{1/q_{3}}\nonumber\\
&\leq \left(\mathbb{E}_{\theta}e^{-\gamma p_{3}\int_{0}^{T}\left(v,q_{\epsilon,v}^{1/2}\left(X^{\epsilon}_{s},\theta\right)\right)^{2}ds}\right)^{1/p_{3}}  \left(\mathbb{E}e^{\gamma C q_{3}\left(1+\sup_{t\in[0,T]}|W_{t}|\right)}\right)^{1/q_{3}}
\label{Eq:CrucialTermToBound2}
\end{align}
where, the first inequality in the last computation used H\"{o}lder inequality with $1/p_{3}+1/q_{3}=1$ and the second inequality used (\ref{Eq:CrucialTermToBound1}).

So, we now need to focus on the term $\mathbb{E}_{\theta}e^{-\gamma p_{3}\int_{0}^{T}\left(v,q_{\epsilon,v}^{1/2}\left(X^{\epsilon}_{s},\theta\right)\right)^{2}ds}$.  Define the vector valued function
\[
d_{\theta_{1},\theta}(x,y)=\sqrt{m_{\theta}(x,y)}\sigma^{-1}(x,y)c_{\theta_{1}}(x,y)
\]
and notice that
\begin{align*}
\left(v,q_{\epsilon,v}^{1/2}\left(X^{\epsilon}_{s},\theta\right)\right)^{2}&=\frac{1}{\epsilon}\int_{\mathcal{Y}}\left\Vert c_{\theta +\sqrt{\epsilon}v}(X^{\epsilon}_{s},y)-c_{\theta}(X^{\epsilon}_{s},y)\right\Vert_{\alpha}^{2}\mu_{\theta}(dy;X^{\epsilon}_{s})\nonumber\\
&=\frac{1}{\epsilon}\int_{\mathcal{Y}}\left\Vert d_{\theta +\sqrt{\epsilon}v,\theta}(X^{\epsilon}_{s},y)-d_{\theta,\theta}(X^{\epsilon}_{s},y)\right\Vert^{2}dy\nonumber
\end{align*}
Using the trivial inequality $a^{2}\geq b^{2}-2|b(a-b)|$, applied with $a=d_{\theta +\sqrt{\epsilon}v,\theta}(X^{\epsilon}_{s},y)-d_{\theta,\theta}(X^{\epsilon}_{s},y)$ and
$b=d_{\theta +\sqrt{\epsilon}v,\theta}(\bar{X}_{s},y)-d_{\theta,\theta}(\bar{X}_{s},y)$ we can write
\begin{align*}
\left(v,q_{\epsilon,v}^{1/2}\left(X^{\epsilon}_{s},\theta\right)\right)^{2}&=\frac{1}{\epsilon}\int_{\mathcal{Y}}\left\Vert d_{\theta +\sqrt{\epsilon}v,\theta}(X^{\epsilon}_{s},y)-d_{\theta,\theta}(X^{\epsilon}_{s},y)\right\Vert^{2}dy\nonumber\\
&\geq \frac{1}{\epsilon}\int_{\mathcal{Y}}\left\Vert d_{\theta +\sqrt{\epsilon}v,\theta}(\bar{X}_{s},y)-d_{\theta,\theta}(\bar{X}_{s},y)\right\Vert^{2}dy\nonumber\\
&\qquad -2\frac{1}{\epsilon}\int_{\mathcal{Y}}\left[\left(\left\Vert d_{\theta +\sqrt{\epsilon}v,\theta}(X^{\epsilon}_{s},y)-d_{\theta +\sqrt{\epsilon}v,\theta}(\bar{X}_{s},y)\right\Vert+\left\Vert d_{\theta,\theta}(X^{\epsilon}_{s},y)-d_{\theta,\theta }(\bar{X}_{s},y)\right\Vert\right)\right.\times\nonumber\\
&\qquad\qquad\left. \times\left\Vert d_{\theta+\sqrt{\epsilon}v,\theta}(\bar{X}_{s},y)-d_{\theta,\theta }(\bar{X}_{s},y) \right\Vert\right] dy
\end{align*}
Hence, we obtain the bound
\begin{align}
&\mathbb{E}_{\theta}e^{-\gamma p_{3}\int_{0}^{T}\left(v,q_{\epsilon,v}^{1/2}\left(X^{\epsilon}_{s},\theta\right)\right)^{2}ds}\leq
e^{-\gamma p_{3}\frac{1}{\epsilon}\int_{0}^{T}\int_{\mathcal{Y}}\left\Vert d_{\theta +\sqrt{\epsilon}v,\theta}(\bar{X}_{s},y)-d_{\theta,\theta}(\bar{X}_{s},y)\right\Vert^{2}dy ds}\times\nonumber\\
&\qquad\qquad \times \mathbb{E}_{\theta}\left[e^{2\gamma p_{3}\frac{1}{\epsilon}\int_{0}^{T}\int_{\mathcal{Y}}\left( \left\Vert d_{\theta+\sqrt{\epsilon}v,\theta}(\bar{X}_{s},y)-d_{\theta,\theta}(\bar{X}_{s},y) \right\Vert\left\Vert d_{\theta +\sqrt{\epsilon}v,\theta}(X^{\epsilon}_{s},y)-d_{\theta +\sqrt{\epsilon}v,\theta}(\bar{X}_{s},y)\right\Vert\right)dy ds}\right.\nonumber\\
&\qquad\qquad \qquad\left.\times e^{2\gamma p_{3}\frac{1}{\epsilon}\int_{0}^{T}\int_{\mathcal{Y}}\left(
\left\Vert d_{\theta+\sqrt{\epsilon}v,\theta}(\bar{X}_{s},y)-d_{\theta,\theta}(\bar{X}_{s},y) \right\Vert \left\Vert d_{\theta,\theta}(X^{\epsilon}_{s},y)-d_{\theta,\theta }(\bar{X}_{s},y)\right\Vert \right)dy ds}\right]\label{Eq:BoundForTerm_0}
\end{align}

Notice that the assumption of uniform positive definiteness of the Fisher information matrix $I(\theta)$ guarantees that
\begin{align}
\int_{0}^{T}\left(v,q^{1/2}(\bar{X}_{s},\theta)\right)^{2}ds=(I(\theta)v,v)&\geq \left\Vert v\right\Vert^{2}\inf_{\theta\in\Theta}\inf_{|\lambda|=1}\left(I(\theta)\lambda,\lambda\right)\nonumber\\
&\geq \left\Vert v\right\Vert^{2}c_{0}\nonumber
\end{align}

So, as  $\left\Vert \sqrt{\epsilon}v \right\Vert\rightarrow 0$ we will have
\begin{align}
&\frac{1}{\epsilon}\int_{0}^{T}\int_{\mathcal{Y}}\left\Vert d_{\theta +\sqrt{\epsilon}v,\theta}(\bar{X}_{s},y)-d_{\theta,\theta}(\bar{X}_{s},y)\right\Vert^{2}dy ds\nonumber\\
&\qquad=\frac{1}{\epsilon}\int_{0}^{T}\int_{\mathcal{Y}}\left\Vert c_{\theta +\sqrt{\epsilon}v}(\bar{X}_{s},y)-c_{\theta}(\bar{X}_{s},y)\right\Vert_{\alpha}^{2}\mu_{\theta}(dy;\bar{X}_{s}) ds\nonumber\\
&\qquad=\int_{0}^{T}\int_{\mathcal{Y}}\left\Vert \int_{0}^{1}\left(v,\nabla_{\theta}c_{\theta +\sqrt{\epsilon}v h}(\bar{X}_{s},y)\right)dh\right\Vert_{\alpha}^{2}\mu_{\theta}(dy;\bar{X}_{s})ds\nonumber\\
&\qquad=\int_{0}^{T}\int_{\mathcal{Y}}\left\Vert \left(v,\nabla_{\theta}c_{\theta}(\bar{X}_{s},y)\right)\right\Vert_{\alpha}^{2}\mu_{\theta}(dy;\bar{X}_{s})ds\nonumber\\
&\quad\qquad+\int_{0}^{T}\int_{\mathcal{Y}}\left\Vert \int_{0}^{1}\left(v,\nabla_{\theta}c_{\theta +\sqrt{\epsilon}v h}(\bar{X}_{s},y)-\nabla_{\theta}c_{\theta}(\bar{X}_{s},y)\right)dh\right\Vert_{\alpha}^{2}\mu_{\theta}(dy;\bar{X}_{s})ds+o(\left\Vert  v \right\Vert^{2} )\nonumber\\
&\qquad=\int_{0}^{T}\left(v,q^{1/2}(\bar{X}_{s},\theta)\right)^{2}ds+o(\left\Vert   v \right\Vert^{2} )\nonumber\\
&\qquad\geq \left\Vert  v\right\Vert^{2}\left(c_{0}+o(1)\right).\label{Eq:BoundForTerm2.3_0}
\end{align}
The assumed uniform boundedness of $\sigma^{-1}c_{\theta}$, the fact that $m_{\theta}$ is a density and the lower bound from (\ref{Eq:BoundForTerm2.3_0}) mean that there exist constants $C_{2},C_{3}$ that may depend on $\tilde{\Theta}$ such that
\begin{equation*}
C_{2}\left\Vert v\right\Vert^{2}\leq \frac{1}{\epsilon}\int_{0}^{T}\int_{\mathcal{Y}}\left\Vert d_{\theta +\sqrt{\epsilon}v,\theta}(\bar{X}_{s},y)-d_{\theta,\theta}(\bar{X}_{s},y)\right\Vert^{2}dy ds\leq C_{3}\left\Vert v\right\Vert^{2}\label{Eq:BoundForTerm2.3_1}
\end{equation*}

Moreover, by Cauchy-Schwartz inequality, we also have that
\begin{align}
&\left(\int_{0}^{T}\int_{\mathcal{Y}}\left(
\left\Vert d_{\theta+\sqrt{\epsilon}v,\theta}(\bar{X}_{s},y)-d_{\theta,\theta}(\bar{X}_{s},y) \right\Vert\left\Vert d_{\theta,\theta}(X^{\epsilon}_{s},y)-d_{\theta,\theta }(\bar{X}_{s},y)\right\Vert\right)dy ds\right)^{2}\nonumber\\
&\leq
\int_{0}^{T}\int_{\mathcal{Y}}\left\Vert d_{\theta+\sqrt{\epsilon}v,\theta}(\bar{X}_{s},y)-d_{\theta,\theta}(\bar{X}_{s},y) \right\Vert^{2}dyds \int_{0}^{T}\int_{\mathcal{Y}}\left\Vert d_{\theta,\theta}(X^{\epsilon}_{s},y)-d_{\theta,\theta }(\bar{X}_{s},y)\right\Vert^{2} dy ds\nonumber\\
&\leq C_{3}\epsilon\left\Vert v\right\Vert^{2} \int_{0}^{T}\int_{\mathcal{Y}}\left\Vert d_{\theta,\theta}(X^{\epsilon}_{s},y)-d_{\theta,\theta }(\bar{X}_{s},y)\right\Vert^{2}dy ds\nonumber\\
&\leq C_{3}C_{4}\epsilon \left\Vert v\right\Vert^{2} \int_{0}^{T} \left\Vert X^{\epsilon}_{s}-\bar{X}_{s}\right\Vert^{2}ds\nonumber\\
&\leq C_{3}C_{4}T\epsilon \left\Vert v\right\Vert^{2} \sup_{t\in[0,T]} \left\Vert X^{\epsilon}_{t}-\bar{X}_{t}\right\Vert^{2}\label{Eq:BoundForTerm2.3_2a}
\end{align}
To derive the inequality before the last one, we used the Lipschitz continuity in $x$ of the function $d_{\theta,\theta}(x,y)$, with a Lipschitz constant $C_{4}$ that may depend on $\tilde{\Theta}$. To continue, we need to bound from above the quantity $\sup_{t\in[0,T]} \left\Vert X^{\epsilon}_{t}-\bar{X}_{t}\right\Vert^{2}$. For this purpose, we set $\bar{c}_{\theta}(x)=\int_{\mathcal{Y}}c_{\theta}(x,y)\mu_{\theta}(dy;x)$ and write
\begin{align}
X^{\epsilon}_{t}-\bar{X}_{t}&=\int_{0}^{t}c_{\theta}\left(X^{\epsilon}_{s},\frac{X^{\epsilon}_{s}}{\delta}\right)ds-\int_{0}^{t}\bar{c}_{\theta}\left(\bar{X}_{s}\right)ds+\sqrt{\epsilon} \int_{0}^{t}\sigma\left(X^{\epsilon}_{s},\frac{X^{\epsilon}_{s}}{\delta}\right)dW_{s}\nonumber\\
&=\int_{0}^{t}\left[c_{\theta}\left(X^{\epsilon}_{s},\frac{X^{\epsilon}_{s}}{\delta}\right)-\bar{c}_{\theta}\left(X^{\epsilon}_{s}\right)\right]ds+\int_{0}^{t}\left[\bar{c}_{\theta}\left(X^{\epsilon}_{s}\right)-\bar{c}_{\theta}\left(\bar{X}_{s}\right)\right]ds+\sqrt{\epsilon} \int_{0}^{t}\sigma\left(X^{\epsilon}_{s},\frac{X^{\epsilon}_{s}}{\delta}\right)dW_{s}\nonumber
\end{align}
Thus, we obtain
\begin{align}
\left\Vert X^{\epsilon}_{t}-\bar{X}_{t}\right\Vert^{2}
&\leq 2^{3}\left\{\left\Vert \int_{0}^{t}\left[c_{\theta}\left(X^{\epsilon}_{s},\frac{X^{\epsilon}_{s}}{\delta}\right)-\bar{c}_{\theta}\left(X^{\epsilon}_{s}\right)\right]ds\right\Vert^{2}+\int_{0}^{t}\left\Vert\bar{c}_{\theta}\left(X^{\epsilon}_{s}\right)-\bar{c}_{\theta}\left(\bar{X}_{s}\right)\right\Vert^{2}ds\right.\nonumber\\
&\qquad\qquad\left.+\epsilon \left\Vert\sigma\right\Vert^{2}\sup_{s\in[0,t]}\left\Vert W_{s}\right\Vert^{2}\right\}\nonumber\\
&\leq 2^{3}\left\{ \left\Vert \int_{0}^{t}\left[c_{\theta}\left(X^{\epsilon}_{s},\frac{X^{\epsilon}_{s}}{\delta}\right)-\bar{c}_{\theta}\left(X^{\epsilon}_{s}\right)\right]ds\right\Vert^{2}+C_{5}\int_{0}^{t}\left\Vert X^{\epsilon}_{s}-\bar{X}_{s}\right\Vert^{2}ds+\epsilon \left\Vert\sigma\right\Vert^{2}\sup_{s\in[0,t]}\left\Vert W_{s}\right\Vert^{2} \right\}
\label{Eq:BoundForTerm2.3_2}
\end{align}
In the last inequality, we used the Lipschitz continuity of $\bar{c}_{\theta}$ with a Lipschitz constant $C_{5}$ that may depend on $\tilde{\Theta}$. Let us explain now how the term $\left\Vert \int_{0}^{t}\left[c_{\theta}\left(X^{\epsilon}_{s},\frac{X^{\epsilon}_{s}}{\delta}\right)-\bar{c}_{\theta}\left(X^{\epsilon}_{s}\right)\right]ds\right\Vert^{2}$ can be treated.  By considering the solution to an auxiliary PDE  problem analogous to (\ref{Eq:PDE_CLT}) with right hand side replaced by $c_{\theta}(x,y)-\bar{c}_{\theta}(x)$, we get  (similarly to (\ref{Eq:PDE_CLT2})) that
\[
\left\Vert \int_{0}^{t}\left[c_{\theta}\left(X^{\epsilon}_{s},\frac{X^{\epsilon}_{s}}{\delta}\right)-\bar{c}_{\theta}\left(X^{\epsilon}_{s}\right)\right]ds\right\Vert\leq C_{6}\left(1+\frac{\delta}{\sqrt{\epsilon}} \sup_{s\in[0,t]}\left\Vert W_{s}\right\Vert\right)
\]
For some constant $C_{6}$ that may depend on $\tilde{\Theta}$. Thus putting things together, (\ref{Eq:BoundForTerm2.3_2}) takes the form
\begin{align}
\left\Vert X^{\epsilon}_{t}-\bar{X}_{t}\right\Vert^{2}
&\leq C_{7}\left\{ 1+ \int_{0}^{t}\left\Vert X^{\epsilon}_{s}-\bar{X}_{s}\right\Vert^{2}ds+\left(\epsilon+\frac{\delta^{2}}{\epsilon} \right)\sup_{s\in[0,t]}\left\Vert W_{s}\right\Vert^{2} \right\}
\label{Eq:BoundForTerm2.3_3}
\end{align}
and by Grownwall inequality, we can conclude that there exists a constant $C_{8}$, that may depend on $\tilde{\Theta}$, such that
\begin{align}
\sup_{t\in[0,T]}\left\Vert X^{\epsilon}_{t}-\bar{X}_{t}\right\Vert
&\leq C_{8}\sqrt{\epsilon+\frac{\delta^{2}}{\epsilon}}\sup_{t\in[0,T]}\left\Vert W_{t}\right\Vert
\label{Eq:BoundForTerm2.3_4a}
\end{align}
Coming back to (\ref{Eq:BoundForTerm2.3_2a}), we have obtained
\begin{align}
&\frac{1}{\epsilon}\int_{0}^{T}\int_{\mathcal{Y}}\left(
\left\Vert d_{\theta+\sqrt{\epsilon}v,\theta}(\bar{X}_{s},y)-d_{\theta,\theta}(\bar{X}_{s},y) \right\Vert\left\Vert d_{\theta,\theta}(X^{\epsilon}_{s},y)-d_{\theta,\theta }(\bar{X}_{s},y)\right\Vert\right)dy ds\nonumber\\
&\qquad\qquad\leq \sqrt{C_{3}C_{4}}C_{8}\left\Vert v\right\Vert  \sup_{t\in[0,T]}\left\Vert W_{t}\right\Vert\label{Eq:BoundForTerm2.3_4}
\end{align}
Set $C_{9}=\sqrt{C_{3}C_{4}T}C_{8}$. Putting (\ref{Eq:BoundForTerm2.3_0}) and (\ref{Eq:BoundForTerm2.3_4}) together and recalling that $v=I^{-1/2}(\theta)u$, the bound (\ref{Eq:BoundForTerm_0}) becomes
\begin{align}
\mathbb{E}_{\theta}e^{-\gamma p_{3}\int_{0}^{T}\left(v,q_{\epsilon,v}^{1/2}\left(X^{\epsilon}_{s},\theta\right)\right)^{2}ds}&\leq
e^{-\gamma p_{3}C_{2}\left\Vert u \right\Vert^{2}} \mathbb{E}_{\theta}\left[e^{4\gamma p_{3}C_{9}\left\Vert u\right\Vert \sup_{t\in[0,T]}\left\Vert W_{t}\right\Vert}\right]\nonumber\\
&\leq
e^{-\gamma p_{3}C_{2}\left\Vert u \right\Vert^{2}} \left[ 1+4\gamma p_{3} C_{9} \left\Vert u\right\Vert \sqrt{8\pi T} e^{8\gamma^{2}\left(p_{3}C_{9}T\right)^{2} \left\Vert u\right\Vert^{2}}\right],
\label{Eq:BoundForTerm_1}
\end{align}
where the last inequality used Lemma 1.14 by Kutoyants, \cite{Kutoyants1994}. Now, we have all the necessary ingredients in order to continue the bound of (\ref{Eq:CrucialTermToBound}). In particular, using (\ref{Eq:CrucialTermToBound2}), (\ref{Eq:CrucialTermToBound}) gives
\begin{align}
\mathbb{E}_{\theta}e^{pM_{\epsilon}(\theta,u)}
&\leq \left(\mathbb{E}_{\theta}e^{-\gamma\int_{0}^{T}\left\Vert \Delta c_{\theta}\right\Vert_{\alpha}^{2}\left(X^{\epsilon}_{s},\frac{X^{\epsilon}_{s}}{\delta}\right)ds}\right)^{(q-p^{2})/p}\nonumber\\
&\leq \left(\mathbb{E}_{\theta}e^{-\gamma p_{3}\int_{0}^{T}\left(v,q_{\epsilon,v}^{1/2}\left(X^{\epsilon}_{s},\theta\right)\right)^{2}ds}\right)^{(q-p^{2})/(pp_{3})}  \left(\mathbb{E}e^{\gamma q_{3}C\left(1+\sup_{t\in[0,T]}|W_{t}|\right)}\right)^{(q-p^{2})/(pq_{3})}\label{Eq:CrucialTermToBound4}
\end{align}

Choosing $p,q$ such that $\gamma=\frac{p_{3}C_{2}}{16 \left(p_{3}C_{9}T\right)^{2}}$ and using the inequality $1+x\leq e^{x}$, we then obtain from (\ref{Eq:BoundForTerm_1})
\begin{align}
\left(\mathbb{E}_{\theta}e^{-\gamma p_{3}\int_{0}^{T}\left(v,q_{\epsilon,v}^{1/2}\left(X^{\epsilon}_{s},\theta\right)\right)^{2}ds}\right)^{(q-p^{2})/(pp_{3})}&\leq
e^{-\frac{C_{2}}{2}\left\Vert u \right\Vert^{2}}  \left[ 1+4\gamma p_{3} C_{9} \left\Vert u\right\Vert \sqrt{8\pi T} \right]^{(q-p^{2})/(pp_{3})}\nonumber\\
&\leq e^{-\frac{C_{2}}{2}\left\Vert u \right\Vert^{2} +\frac{q-p^{2}}{p}4\gamma  C_{9} \left\Vert u\right\Vert \sqrt{8\pi T} }
\label{Eq:CrucialTermToBound5}
\end{align}
So, (\ref{Eq:CrucialTermToBound4}) and (\ref{Eq:CrucialTermToBound5}) give
\begin{align*}
\mathbb{E}_{\theta}e^{pM_{\epsilon}(\theta,u)}
&\leq e^{-\frac{C_{2}}{2}\left\Vert u \right\Vert^{2} +\frac{q-p^{2}}{p}4\gamma  C_{9} \left\Vert u\right\Vert \sqrt{8\pi T} } \left(\mathbb{E}e^{\gamma q_{3}C\left(1+\sup_{t\in[0,T]}|W_{t}|\right)}\right)^{(q-p^{2})/(pq_{3})}\label{Eq:CrucialTermToBound6}
\end{align*}
The right hand side of the last inequality defines our function $g(\left\Vert u\right\Vert)$, which certainly enjoys the property
\[
\lim_{u\rightarrow\infty}u^{n}e^{-g(\left\Vert u\right\Vert)}=0,\quad \forall n\in\mathbb{N}
\]
This concludes the proof of the lemma.
\end{proof}


\end{document}